\documentclass[11pt, twoside]{article}
\pdfoutput=1

\usepackage{graphicx}
\usepackage[caption=false]{subfig}
\usepackage{amsmath}
\usepackage{amssymb,amsfonts}
\usepackage{amsthm}
\usepackage{bm}
\usepackage{mathrsfs}
\usepackage{amssymb}
\usepackage{multirow}

\newcommand{\norm}[1]{\left\lVert#1\right\rVert}

\usepackage[margin=1in]{geometry}

\usepackage{algorithm}
\usepackage{algorithmic}

\numberwithin{equation}{section}

\theoremstyle{definition}
\newtheorem{theorem}{Theorem}

\newtheorem{lemma}{Lemma}
\newtheorem{corollary}{Corollary}

\newtheorem{remark}{Remark}

\usepackage{cite}
\usepackage{hyperref}
\usepackage[nameinlink]{cleveref}

\usepackage{fancyhdr}
\begin{document}
	
	\title{Interpolatory HDG Method for Parabolic Semilinear PDEs}

\author{Bernardo Cockburn%
	\thanks{School of Mathematics, University of Minnesota, Minneapolis, MN (\mbox{cockburn@math.umn.edu}).}
	\and
	John~R.~Singler%
	\thanks{Department of Mathematics and Statistics, Missouri University of Science and Technology, Rolla, MO (\mbox{singlerj@mst.edu}, \mbox{ywzfg4@mst.edu}).}
	\and
	Yangwen Zhang%
	\footnotemark[2]
}

\date{February 10, 2018}

\maketitle

\begin{abstract}
	We propose the interpolatory hybridizable discontinuous Galerkin (Interpolatory HDG) method for a class of scalar parabolic semilinear PDEs.  The Interpolatory HDG method uses an interpolation procedure to efficiently and accurately approximate the nonlinear term.  This procedure avoids the numerical quadrature typically required for the assembly of the global matrix at each iteration in each time step, which is a computationally costly component of the standard HDG method for nonlinear PDEs.  Furthermore, the Interpolatory HDG interpolation procedure yields simple explicit expressions for the nonlinear term and Jacobian matrix, which leads to a simple unified implementation for a variety of nonlinear PDEs.  For a globally-Lipschitz nonlinearity, we prove that the Interpolatory HDG method does not result in a reduction of the order of convergence.  We display 2D and 3D numerical experiments to demonstrate the performance of the method.
\end{abstract}

\section{Introduction}

\label{introduction}
We consider the following class of scalar parabolic semilinear PDEs on a Lipschitz polyhedral domain $\Omega\subset \mathbb R^d $, $ d\ge 2$, with boundary $\partial\Omega$:
\begin{equation}\label{semilinear_pde}
\begin{split}
\partial_tu-\Delta u+ F(\nabla u,u)&= f \quad  \mbox{in} \; \Omega\times(0,T],\\
u&=0 \quad  \mbox{on} \; \partial\Omega\times(0,T],\\
u(\cdot,0)&=u_0~~\mbox{in}\ \Omega.
\end{split}
\end{equation}
A challenge in the simulation of nonlinear PDEs is to reduce computational cost while preserving accuracy. To this end, a large amount of research in numerous aspects of the simulation of PDEs has been performed. These efforts include attempts to reduce computational cost by improving algorithmic efficiency, developing parallel computing schemes, and using interpolatory finite element (Interpolatory FE) techniques.

The Interpolatory FE method, also known as {\em product approximation}, the {\em group finite element method}, or finite elements with interpolated coefficients, was originally proposed by Douglas and Dupont for solving semilinear parabolic problems in \cite{MR0502033}. The technique was later rediscovered by Christie et al.\ \cite{MR641309} and then by Fletcher \cite{MR798845,MR702221}. In recent years,  the Interpolatory FE method has been used as an alternative FE method for nonlinear elliptic problems \cite{MR967844,MR731213,MR2752869,MR2273051,MR2112661}, nonlinear parabolic problems \cite{MR1030644,MR1172090,MR973559,MR3178584,MR2294957},  nonlinear hyperbolic problems \cite{MR1068202,MR2391691}, and model order reduction methods \cite{MR3403707,MR2587427}. This approach consists in replacing the nonlinear function by its interpolant in the finite dimensional space.  This simple change leads to an important benefit: the numerical quadrature for the nonlinear term is computed once before time integration, which leads  to a simplified implementation and a substantial reduction in computational cost. Furthermore, the Interpolatory FE method achieves the same convergence rates as the standard FE discretization of the PDEs.

However, to the best knowledge of the authors, the Interpolatory FE method is applicable for the problem above only if the nonlinear term $F(\nabla u, u)$ can be  written in a special ``grouped'' form, i.e., when there exist functions $\bm G$ and $H$ such that  $F(\nabla u, u) = \nabla \cdot \bm G(u) + H(u)$.  For other types of PDEs and PDE systems, the grouped form can be slightly more general. However, not all nonlinearities of interest can be written in the {\em special grouped} form which results in the limited applicability of the Interpolatory FE method.

We propose a new method to approximate the solution: the interpolatory hybridizable discontinuous Galerkin (Interpolatory HDG) method.  Specifically, we show that the interpolation idea of the Interpolatory FE method can be extended to the hybridizable discontinuous Galerkin (HDG) method, see a recent review of these methods in \cite{CockburnDurham16}, and  that the resulting Interpolatory HDG method can successfully be applied to the general nonlinear term $F(\nabla u, u)$.  For more information about HDG methods for nonlinear PDEs, see, e.g., \cite{NguyenPeraireCockburnHDGAIAAINS10,PeraireNguyenCockburnHDGAIAACNS10,NguyenPeraireCockburn11,NguyenPeraireCM12,MoroNguyenPeraireSCL12,MR3626531,NguyenPeraireCockburnEDG15,KabariaLewCockburn15,GaticaSequeira15,MR2558780,UeckermannLermusiaux16,MR3463051}.


The paper is organized as follows. We introduce the semidiscrete standard and the Interpolatory HDG methods in Section \ref{sec:HDG} and discuss their  implementation in detail for a simple time-discretization in Section \ref{sec:HDG2}. We then analyze the semidiscrete Interpolatory HDG method in Section \ref{sec:analysis} and prove optimal convergence rates for a globally Lipschitz nonlinearity.  Finally, we illustrate the performance of the Interpolatory HDG method in Section \ref{sec:numerics} with 2D and 3D numerical experiments.

\section{Semidiscrete Standard and Interpolatory HDG Formulations}
\label{sec:HDG}



\subsection{Notation}

To introduce the space-discretization by the HDG methods, we first set some notation; we follow \cite{MR2485455}, where the HDG methods were introduced in the framework of linear, steady-state diffusion.

Let $\mathcal{T}_h$ be a collection of disjoint simplexes $K$ that partition $\Omega$.  Let $\partial \mathcal{T}_h$ denote the set $\{\partial K: K\in \mathcal{T}_h\}$. For an element $K$ of the collection  $\mathcal{T}_h$, let $e = \partial K \cap \Gamma$ denote the boundary face of $ K $ if the $d-1$ Lebesgue measure of $e$ is non-zero. For two elements $K^+$ and $K^-$ of the collection $\mathcal{T}_h$, let $e = \partial K^+ \cap \partial K^-$ denote the interior face between $K^+$ and $K^-$ if the $d-1$ Lebesgue measure of $e$ is non-zero. Let $\varepsilon_h^o$ and $\varepsilon_h^{\partial}$ denote the sets of interior and boundary faces, respectively, and let $\varepsilon_h$ denote the union of  $\varepsilon_h^o$ and $\varepsilon_h^{\partial}$. We finally set
\begin{align*}
(w,v)_{\mathcal{T}_h} := \sum_{K\in\mathcal{T}_h} (w,v)_K,   \quad\quad\quad\quad\left\langle \zeta,\rho\right\rangle_{\partial\mathcal{T}_h} := \sum_{K\in\mathcal{T}_h} \left\langle \zeta,\rho\right\rangle_{\partial K},
\end{align*}
where, when $D\subset\mathbb{R}^d$, $(\cdot,\cdot)_D$  denotes the $L^2(D)$ inner product and,
when $\Gamma$ is the union of subsets of  $\mathbb{R}^{d-1}$, $\langle \cdot, \cdot\rangle_{\Gamma} $ denotes the $L^2(\Gamma)$ inner product.

Let $\mathcal{P}^k(K)$ denote the set of polynomials of degree at most $k$ on a domain $K$.  We consider the discontinuous finite element spaces
\begin{subequations}\label{spaces}
	\begin{alignat}{4}
	\bm{V}_h  &:= \{\bm{v}\in [L^2(\Omega)]^d: &&\bm{v}|_{K}\in [\mathcal{P}^k(K)]^d, &&\forall K\in \mathcal{T}_h\},\\
	{W}_h  &:= \{{w}\in L^2(\Omega): &&{w}|_{K}\in \mathcal{P}^{k}(K), &&\forall K\in \mathcal{T}_h\},\\
	{M}_h  &:= \{{\mu}\in L^2(\mathcal{\varepsilon}_h): &&{\mu}|_{e}\in \mathcal{P}^k(e), &&\forall e\in \varepsilon_h,&&\mu|_{\varepsilon_h^\partial} = 0\},
	\end{alignat}
\end{subequations}
for the flux variables, scalar variables, and trace variables, respectively.  
%
Note that $M_h$ consists of functions which are continuous inside the faces (or edges) $e\in \varepsilon_h$ and discontinuous at their borders. Also, for $ w \in W_h $ and $ \bm r \in \bm V_h $, let $ \nabla w $ and $ \nabla \cdot \bm r $ denote the gradient of $ w $ and divergence and $ \bm r $ applied piecewise on each element $K\in \mathcal T_h$.

\subsection{Semidiscrete Standard HDG Formulation}

The HDG method introduces the flux $\bm q = -\nabla u$, and rewrites the semilinear PDE \eqref{semilinear_pde} in the mixed form
\begin{subequations}\label{semilinear_pde_mixed}
	\begin{align}
	(\bm{q},\bm{r})-(u,\nabla\cdot\bm{r})+\left\langle u, \bm{r}\cdot \bm{n}\right\rangle&=0, \label{mixed_a}\\
	(\partial_tu, w) + (\nabla\cdot \bm{q}, w) + (F(-\bm q,u),w) &= (f,w),  \label{mixed_b}\\
	(u(\cdot,0),w)&=(u_0,w), \label{mixed_c}
	\end{align}
\end{subequations}
for all $(\bm{r},w)\in H(\text{div},\Omega)\times L^2(\Omega)$.

To approximate the solution of the mixed weak form \eqref{semilinear_pde_mixed} of \eqref{semilinear_pde}, the standard HDG method seeks an approximate flux $\bm q_h \in \bm V_h$, primary variable $u_h \in W_h$,  and numerical boundary trace $\widehat u_h\in M_h$ satisfying
\begin{subequations}\label{semi_discretion_standard}
	\begin{align}
	(\bm{q}_h,\bm{r})_{\mathcal{T}_h}-(u_h,\nabla\cdot \bm{r})_{\mathcal{T}_h}+\left\langle\widehat{u}_h,\bm r\cdot\bm n \right\rangle_{\partial{\mathcal{T}_h}} &= 0, \label{semi_discretion_standard_a}\\
	(\partial_tu_h,w)_{\mathcal T_h}-(\bm{q}_h,\nabla w)_{\mathcal{T}_h}+\left\langle\widehat{\bm{q}}_h\cdot \bm{n},w\right\rangle_{\partial{\mathcal{T}_h}} +  ( F(-\bm q_h, u_h),w)_{\mathcal{T}_h}&= (f,w)_{\mathcal{T}_h},\label{semi_discretion_standard_b}\\
	\left\langle\widehat{\bm{q}}_h\cdot \bm{n}, \mu\right\rangle_{\partial{\mathcal{T}_h}\backslash\varepsilon^{\partial}_h} &=0,\label{semi_discretion_standard_c}
	\end{align}
\end{subequations}
for all $(\bm r,w,\mu)\in \bm V_h\times W_h\times M_h$. Here the numerical trace for the flux on $\partial\mathcal T_h$ is defined by
\begin{align*}
\widehat{\bm q}_h\cdot\bm n=\bm q_h\cdot\bm n+\tau( u_h-\widehat u_h),
\end{align*}
where $\tau$ is positive stabilization function defined on $\partial\mathcal T_h$.  The initial conditions are discretized as
\begin{align}\label{semi_discretion_standard_d}
u_h(\cdot,0) &= P u_0,
\end{align}
where $ P $ is a projection into $ W_h $.

\subsection{Semidiscrete Interpolatory HDG Formulation}
\label{sec:GHDG_semidiscrete}

To define the Interpolatory HDG space discretization of \eqref{semilinear_pde_mixed}, we first define the operator $\mathcal{I}_h$  we use to approximate the nonlinear term $F(-\bm q_h,u_h)$.

For an element $K \in \mathcal{T}_h $, let $\{\xi^K_j\}_{j=1}^{\ell_K}$ denote the FE nodal points corresponding to the nodal basis functions $\{\phi^K_j\}_{j=1}^{\ell_K}$ for $ W_h(K) $, i.e., $\phi^K_j(\xi^K_i) = \delta_{ij}$, where $\delta_{ij}$ is the Kronecker delta symbol, and $ W_h(K) = \mathrm{span}\{\phi^K_j\}_{j=1}^{\ell_K}$.  For $ g \in C(\bar{K}) $, define $ \mathcal{I}^K_h g \in W_h(K) $ by
$$
[\mathcal{I}^K_h g](x) = \sum_{j=1}^{\ell_K} g(\xi^K_j) \phi^K_j(x)  \quad  \mbox{for all $ x \in K $.}
$$
Note that this indeed defines an interpolation operator on $ K $ since $\phi^K_j(\xi^K_i) = \delta_{ij}$ implies $ [\mathcal{I}^K_h g](\xi^K_i) = g(\xi^K_i) $.

Next, we extend the definition to the set of square integrable elementwise continuous functions
$$
Z = \{ g \in L^2(\Omega) : g|_{K} \in C(\bar{K}), \forall K \in \mathcal{T}_h \}.
$$
For $ g \in Z $, define $ \mathcal{I}_h g \in W_h $ to equal the above interpolation $ \mathcal{I}^K_h g $ in $ W_h(K) $ on each element $ K $.  Note that $ \mathcal I_h g $ may be discontinuous along the faces (or edges).  Furthermore, note that $\mathcal{I}_h$ is not an interpolation operator since multiple discontinuous basis functions in $ W_h $ correspond to a single interior nodal point in $ \Omega $.  However, since $\mathcal{I}_h$ is an interpolation operator when restricted to an individual element, we call $\mathcal{I}_h$ an elementwise interpolation operator.

Now that we have defined the operator $\mathcal{I}_h$, we present the Interpolatory HDG formulation of \eqref{semilinear_pde_mixed}.  It is obtained by replacing, in the equation \eqref{semi_discretion_standard_b} of the HDG formulation, the nonlinear term $F(-\bm q_h, u_h)$ by  the elementwise interpolation $\mathcal  I_h F(-\bm q_h, u_h)$. We thus obtain, instead of
\eqref{semi_discretion_standard_b}, the equation
\begin{align}\label{semi_discretion_group_b}
(\partial_tu_h,w)_{\mathcal T_h}-(\bm{q}_h,\nabla w)_{\mathcal{T}_h}+\left\langle\widehat{\bm{q}}_h\cdot \bm{n},w\right\rangle_{\partial{\mathcal{T}_h}} +  ( \mathcal I_h F(-\bm q_h ,u_h),w)_{\mathcal{T}_h} &= (f,w)_{\mathcal{T}_h}
\end{align}
for $ w \in W_h $.

Before discussing the implementation details, we briefly discuss the computational advantage of the elementwise interpolation.  We consider the 2D case here; the 3D case is similar.  Let $ W_h = \mathrm{span}\{\phi_j\}_{j=1}^{\ell} $, where each $ \phi_j $ is a nodal FE basis function when restricted to some element.  Let $ u_h $ be represented as
$$
u_h(t) = \sum_{j=1}^{N_1} \gamma_j(t) \phi_j.
$$
Functions in the space $ \bm V_h $ can be represented componentwise using the same basis functions.  Let $ \bm q_h = [ q_{h,1}, q_{h,2} ]^T $ be represented as
$$
q_{h,1}(t) = \sum_{j=1}^{N_1} \alpha_j(t) \phi_j,  \quad	  q_{h,2}(t) = \sum_{j=1}^{N_1} \beta_j(t) \phi_j.
$$
For $ w \in W_h $, the nonlinearity in the Interpolatory HDG method takes the form
\begin{align*}
( \mathcal I_h F(-\bm q_h ,u_h),w)_{\mathcal{T}_h} &= \sum_{K \in \mathcal T_h} ( \mathcal I^K_h F(-q_{h,1},-q_{h,2},u_h),w)_{K}\\
&= \sum_{j=1}^{N_1} F(-\alpha_j,-\beta_j,\gamma_j) \, ( \phi_j,w)_{\mathcal{T}_h}.
\end{align*}

In the computation, we take test functions $ w = \phi_i $ and the approximate nonlinearity is quickly evaluated by multiplying a sparse matrix times the vector $ \mathcal{F} = [ F(-\alpha_j,-\beta_j,\gamma_j) ] $.  We provide more implementation details in  Section \ref{sec:GHDG_implementation}.
\begin{remark}
	If  $W_h$ and $\bm V_h$  are the space of the lowest order, i.e., $k=0$,
	then we have $\mathcal I_h F(-\bm q_h,u_h) = F(-\bm q_h,u_h)$ for any $\bm q_h\in \bm V_h$ and $u_h\in W_h$.  This implies that the standard HDG is equivalent to the  Interpolatory HDG when $k=0$.
\end{remark}

\section{Standard and Interpolatory HDG Implementation}
\label{sec:HDG2}

Next we  discuss the main implementation details for the standard and the Interpolatory HDG methods by using a simple time discretization approach: backward Euler with a Newton iteration to solve the nonlinear system at each time step.  Implementation details are similar for other fully implicit time stepping methods.  Interpolatory HDG can also be used with other time discretization approaches, such as implicit-explicit methods and adaptive time stepping approaches.  Interpolatory HDG will provide the greatest computational savings for fully implicit methods.  Even for a time stepping method that only solves one linear system per time step, the Interpolatory HDG can be used to avoid the numerical quadrature required by standard HDG at each time step.

To compare  the standard HDG with the Interpolatory HDG, we suppose that the nonlinear term depends on $ u $ only, i.e., $F(-\bm q, u) = F(u)$.  The general case of a general nonlinearity $F(-\bm q, u)$ is treated in Section \ref{sec:GHDG_general_implementation}.  We only give details for the implementation in 2D; the implementation in 3D is similar.

Let  $N$ be a positive integer and define the time step $\Delta t = T/N$. We denote the approximation of $(\bm q_h(t),u_h(t),\widehat u_h(t))$ by $(\bm q^n_h,u^n_h,\widehat u^n_h)$ at the discrete time $t_n = n\Delta t $, for $n = 0,1,2,\ldots,N$.  For both the standard HDG and the Interpolatory HDG, we replace the time derivative $\partial_tu_h$ in \eqref{semi_discretion_standard} by the backward Euler difference quotient 
\begin{align}
\partial^+_tu^n_h = \frac{u^n_h-u^{n-1}_h}{\Delta t}.\label{backward_Euler}
\end{align}
This gives the following fully discrete method: find $(\bm q^n_h,u^n_h,\widehat u^n_h)\in \bm V_h\times W_h\times M_h$ satisfying
\begin{subequations}\label{full_discretion_standard}
	\begin{align}
	(\bm{q}^n_h,\bm{r})_{\mathcal{T}_h}-(u^n_h,\nabla\cdot \bm{r})_{\mathcal{T}_h}+\left\langle\widehat{u}^n_h,\bm{r\cdot n} \right\rangle_{\partial{\mathcal{T}_h}} &= 0,\label{full_discretion_standard_a} \\
	(\partial^+_tu^n_h,w)_{\mathcal T_h}-(\bm{q}^n_h,\nabla w)_{\mathcal{T}_h}+\left\langle\widehat{\bm{q}}^n_h\cdot \bm{n},w\right\rangle_{\partial{\mathcal{T}_h}} + ( F(-\bm q_h^n, u_h^n),w)_{\mathcal{T}_h}&= (f^n,w)_{\mathcal{T}_h},\label{full_discretion_standard_b}\\
	\left\langle\widehat{\bm{q}}^n_h\cdot \bm{n}, \mu\right\rangle_{\partial{\mathcal{T}_h}\backslash\varepsilon^{\partial}_h} &=0,\label{full_discretion_standard_c}\\
	u^0_h &=P  u_0,\label{full_discretion_standard_d}
	\end{align}
\end{subequations}
for all $(\bm r,w,\mu)\in \bm V_h\times W_h\times M_h$ and $n=1,2,\ldots,N$.  In \eqref{full_discretion_standard}, $f^n=f(\cdot,t_n)$ and the numerical trace for the flux on $\partial\mathcal T_h$ is defined by
\begin{align}\label{num_tr_s}
\widehat{\bm q}_h^n\cdot\bm n=\bm q_h^n\cdot\bm n+\tau(u_h^n-\widehat u_h^n).
\end{align}	
The full Interpolatory HDG discretization only changes the nonlinear term $F(-\bm q_h^n, u_h^n)$ in \eqref{full_discretion_standard_b} into the elementwise interpolation $\mathcal  I_h F(-\bm q_h^n, u_h^n)$, i.e., \eqref{full_discretion_standard_b} is replaced by
\begin{align}
(\partial^+_tu^n_h,w)_{\mathcal T_h}-(\bm{q}^n_h,\nabla w)_{\mathcal{T}_h}+\left\langle\widehat{\bm{q}}^n_h\cdot \bm{n},w\right\rangle_{\partial{\mathcal{T}_h}} + (\mathcal I_h F(-\bm q_h^n, u_h^n),w)_{\mathcal{T}_h}&= (f^n,w)_{\mathcal{T}_h}.
\end{align}

\subsection{Standard HDG Implementation}


After substituting \eqref{num_tr_s} into \eqref{full_discretion_standard_a}-\eqref{full_discretion_standard_c} and integrating by parts, we have $(\bm q^n_h,u^n_h,\widehat u^n_h)\in \bm V_h\times W_h\times M_h$ satisfies
\begin{equation}\label{full_discretion_standard_im}
\begin{split}
(\bm{q}^n_h,\bm{r})_{\mathcal{T}_h}-(u^n_h,\nabla\cdot \bm{r})_{\mathcal{T}_h}+\langle\widehat{u}^n_h,\bm{r\cdot n} \rangle_{\partial{\mathcal{T}_h}} &= 0, \\
(\partial^+_tu^n_h,w)_{\mathcal T_h}+(\nabla\cdot\bm{q}^n_h, w)_{\mathcal{T}_h}+\langle\tau(u_h^n - \widehat u_h^n),w\rangle_{\partial{\mathcal{T}_h}} + ( F( u_h^n),w)_{\mathcal{T}_h}&= (f^n,w)_{\mathcal{T}_h},\\
\langle {\bm{q}}^n_h\cdot \bm{n} + \tau(u_h^n - \widehat u_h^n), \mu\rangle_{\partial{\mathcal{T}_h}\backslash\varepsilon^{\partial}_h} &=0,\\
u^0_h &= P u_0,
\end{split}
\end{equation}
for all $(\bm r,w,\mu)\in \bm V_h\times W_h\times M_h$ and $n=1,2,\ldots,N$.

At each time step $t_n $ for $ 1\le n\le N$, given an initial guess  $(\bm q_h^{n,(0)}, u_h^{n,(0)},\widehat u_h^{n,(0)})$, Newton's method generates the sequence $(\bm q_h^{n,(m)}, u_h^{n,(m)},\widehat u_h^{n,(m)})$ for $￼m = 1, 2, 3,\ldots$ by solving the sequence of linear problems
\begin{equation}\label{full_discretion_standard_im2}
\begin{split}
(\bm{q}^{n,(m)}_h,\bm{r})_{\mathcal{T}_h}-(u^{n,(m)}_h,\nabla\cdot \bm{r})_{\mathcal{T}_h}+\langle\widehat{u}^{n,(m)}_h,\bm r \cdot \bm n \rangle_{\partial{\mathcal{T}_h}} &= 0, \\
\frac {1}{\Delta t}({u^{n,(m)}_h-u^{n-1}_h}, w)_{\mathcal T_h}+(\nabla\cdot\bm{q}^{n,(m)}_h, w)_{\mathcal{T}_h} \quad  &\\
+\langle \tau( u^{n,(m)}_h -\widehat u^{n,(m)}_h) ,w\rangle_{\partial{\mathcal{T}_h}} + ( F'(u^{n,(m-1)}_h)u^{n,(m)}_h,w)_{\mathcal{T}_h} \quad  &\\
- ( F'(u^{n,(m-1)}_h)u^{n,(m-1)}_h,w)_{\mathcal{T}_h} + ( F(u^{n,(m-1)}_h),w)_{\mathcal{T}_h}&= (f^n,w)_{\mathcal{T}_h},\\
\langle {\bm{q}}^{n,(m)}_h \cdot \bm{n}+\tau( u^{n,(m)}_h -\widehat u^{n,(m)}_h), \mu\rangle_{\partial{\mathcal{T}_h}\backslash\varepsilon^{\partial}_h} &=0,
\end{split}
\end{equation}
for all $(\bm r,w,\mu)\in \bm V_h\times W_h\times M_h$.  When the iteration stops, we set $u_h^n = u_h^{n,(m) }$.

Assume $\bm{V}_h = \mbox{span}\{\phi_i\}_{i=1}^{N_1} \times \mbox{span}\{\phi_i\}_{i=1}^{N_1}$, $W_h=\mbox{span}\{\phi_i\}_{i=1}^{N_1}$, and $M_h=\mbox{span}\{\psi_i\}_{i=1}^{N_2}$. Then
\begin{equation}\label{expre}
\begin{split}
q^{n,(m)}_{1h}= \sum_{j=1}^{N_1}\alpha_{j}^{n,(m)}\phi_j,  ~	q^{n,(m)}_{2h}= \sum_{j=1}^{N_1}\beta_{j}^{n,(m)}\phi_j,  \\
u^{n,(m)}_h= \sum_{j=1}^{N_1}\gamma_{j}^{n,(m)}\phi_j, ~
\widehat{u}^{n,(m)}_h= \sum_{j=1}^{N_2}\zeta_{j}^{n,(m)}\psi_{j}.
\end{split}
\end{equation}
Substitute \eqref{expre} into \eqref{full_discretion_standard_im2} and use the corresponding  test functions to test \eqref{full_discretion_standard_im2}, respectively, to obtain the matrix equation
\begin{align}\label{system_equation}
\begin{bmatrix}
A_1 & 0 &-A_2  & A_4 \\
0 & A_1 &-A_3  & A_5 \\
A_2^T& A_3^T&A_6 +{\Delta t}^{-1}A_1+A_9^{n,(m)}&-A_7\\
A_4^T& A_5^T&A_7^T&-A_8
\end{bmatrix}
\left[ {\begin{array}{*{20}{c}}
	\bm\alpha^{n,(m)}\\
	\bm\beta^{n,(m)}\\
	\bm\gamma^{n,(m)}\\
	\bm\zeta^{n,(m)}
	\end{array}} \right]
=\left[ {\begin{array}{*{20}{c}}
	0\\
	0\\
	b^n \\
	0
	\end{array}} \right],
\end{align}
where $\bm\alpha^{n,(m)} $, $\bm\beta^{n,(m)} $, $\bm\gamma^{n,(m)}$, and $ \bm\zeta^{n,(m)} $ are the coefficient vectors and
\begin{gather*}
A_1= [(\phi_j,\phi_i )_{\mathcal{T}_h}],  ~  A_2 = [(\phi_j,\frac{\partial\phi_i}{\partial x})_{\mathcal{T}_h}],  ~  A_3= [(\phi_j,\frac{\partial\phi_i}{\partial y})_{\mathcal{T}_h}],~  A_4= [(\psi_j,{\phi}_i n_1)_{\mathcal{T}_h}],\\
A_5= [(\psi_j,{\phi}_i n_2)_{\mathcal{T}_h}], \qquad  A_6 = [\langle \tau\phi_j, \phi_i \rangle_{\partial{{\mathcal{T}_h}}}], \qquad   A_7 =  [\left\langle \tau\psi_j,\phi_i\right\rangle_{\partial\mathcal{T}_h}],\\
A_8= [\left\langle \tau\psi_j,\psi_i\right\rangle_{\partial\mathcal{T}_h}], \  b_1^n = [(f(t_n,\cdot),\phi_i )_{\mathcal{T}_h}],\quad
b_2 =	[(F(u^{n,(m-1)}_h),\phi_i )_{\mathcal{T}_h}],\\
b^n = b_1^n+ \Delta t^{-1} A_1 \bm\gamma^{n-1} + A_9^{n,(m)} \bm\gamma^{n,(m-1)} -b_2^{n,(m)},\\
A_9^{n,(m)} =	[(F'(u^{n,(m-1)}_h)\phi_j,\phi_i )_{\mathcal{T}_h}].
\end{gather*}
We need to perform numerical quadrature to construct the matrix $A_9^{n,(m)}$ at each time $t_n$ and each step in the iteration, and then solve the linear system \eqref{system_equation}.

\subsection{Interpolatory HDG Implementation}
\label{sec:GHDG_implementation}

The full Interpolatory HDG discretization is to find $(\bm q^n_h,u^n_h,\widehat u^n_h)\in \bm V_h\times W_h\times M_h$ such that
\begin{equation}\label{full_discretion_Group_im}
\begin{split}
(\bm{q}^n_h,\bm{r})_{\mathcal{T}_h}-(u^n_h,\nabla\cdot \bm{r})_{\mathcal{T}_h}+\left\langle\widehat{u}^n_h,\bm{r\cdot n} \right\rangle_{\partial{\mathcal{T}_h}} &= 0, \\
(\partial^+_tu^n_h,w)_{\mathcal T_h}+(\nabla\cdot\bm{q}^n_h, w)_{\mathcal{T}_h}+\langle\tau(u_h^n - \widehat u_h^n),w\rangle_{\partial{\mathcal{T}_h}} + ( \mathcal I_h F(u_h^n),w)_{\mathcal{T}_h}&= (f^n,w)_{\mathcal{T}_h},\\
\langle {\bm{q}}^n_h\cdot \bm{n} + \tau(u_h^n - \widehat u_h^n), \mu\rangle_{\partial{\mathcal{T}_h}\backslash\varepsilon^{\partial}_h} &=0,\\
u^0_h &=P  u_0,
\end{split}
\end{equation}
for all $(\bm r,w,\mu)\in \bm V_h\times W_h\times M_h$ and $n=1,2,\ldots,N$.

The only difference between the Interpolatory HDG and the standard HDG is in the nonlinear terms; the linear parts are the same.  As indicated in Section \ref{sec:GHDG_semidiscrete}, once we test using $ w = \phi_i $ we can express the Interpolatory HDG nonlinear term by the matrix-vector product
\begin{align}
[ ( \mathcal I_h F(u_h^n),\phi_i)_{\mathcal{T}_h} ] = A_1 \mathcal F(\bm\gamma^{n}),
\end{align}
where $\mathcal F$ is defined by
\begin{align}
\mathcal F(\bm\gamma^{n}) &= [F(\gamma_1^{n}), F(\gamma_2^{n}),\cdots,F(\gamma_{N_1}^{n})]^T.
\end{align}

Then the system \eqref{full_discretion_Group_im} can be rewritten as
\begin{align}\label{system_equation_group3}
\underbrace{\begin{bmatrix}
	A_1 & 0 &-A_2  & A_4 \\
	0 & A_1 &-A_3  & A_5 \\
	A_2^T& A_3^T&A_6 +{\Delta t}^{-1}A_1&-A_7\\
	A_4^T& A_5^T&A_7^T&-A_8
	\end{bmatrix}}_{M}
\underbrace{\left[ {\begin{array}{*{20}{c}}
		\bm\alpha^{n}\\
		\bm\beta^{n}\\
		\bm\gamma^{n}\\
		\bm\zeta^{n}\\
		\end{array}} \right]}_{\bm x_{n}}+
\underbrace{\left[ {\begin{array}{*{20}{c}}
		0\\
		0\\
		A_1\mathcal F(\bm\gamma^{n})\\
		0
		\end{array}} \right]}_{\mathscr F(\bm x_{n})}
=\underbrace{\left[ {\begin{array}{*{20}{c}}
		0\\
		0\\
		b_1^n+{\Delta t}^{-1}A_1\bm\gamma^{n-1} \\
		0
		\end{array}} \right]}_{\bm b_n},
\end{align}
i.e., 
\begin{align}\label{system_equation_group4}
M\bm x_n + \mathscr F(\bm x_n) = \bm b_n.
\end{align}

To apply Newton's method to solve the nonlinear equations \eqref{system_equation_group4}, define $G:\mathbb R^{3N_1+N_2}\to \mathbb R^{3N_1+N_2}$ by
\begin{align}\label{system_equation_group5}
G(\bm x_n ) = M\bm x_n + \mathscr F(\bm x_n) - \bm b_n.
\end{align}
At each time step $t_n $ for $ 1\le n\le N$, given an initial guess $\bm x_n^{(0)}$ Newton's method generates the sequence $\bm x_n^{(m)}$ for $￼m = 1, 2, 3,\ldots$ by solving the sequence of linear problems
\begin{align}\label{system_equation_group6}
\bm x_n^{(m)} =\bm x_n^{(m-1)} - \left[G'(\bm x_n^{(m-1)})\right]^{-1}G(\bm x_n^{(m-1)}),
\end{align}
where  the Jacobian matrix $G'(\bm x_n^{(m-1)})$ is given by
\begin{align}\label{system_equation_group7}
G'(\bm x_n^{(m-1)}) = M+\mathscr F'(\bm x_n^{(m-1)}).
\end{align}
An excellent property of the interpolatory method is that we can explicitly compute $\mathscr F'(\bm x_n^{(m-1)})$ by
\begin{align*}
\mathscr F'(\bm x_n^{(m-1)}) = \begin{bmatrix}
0 & 0 &0  & 0 \\
0 & 0 &0  & 0 \\
0& 0&A_{10}^{n,(m)}&0\\
0 & 0 &0  & 0 
\end{bmatrix},
\end{align*}
where $ A_{10}^{n,(m)} $ is quickly and easily computed using sparse matrix operations by
\begin{align*}
A_{10}^{n,(m)} = A_1 \, \text{diag}(\mathcal F'(\bm{\gamma}^{n,(m-1)})).
\end{align*}

We can rewrite equation \eqref{system_equation_group6} as
\begin{align}\label{system_equation_group1}
\begin{bmatrix}
A_1 & 0 &-A_2  & A_4 \\
0 & A_1 &-A_3  & A_5 \\
A_2^T& A_3^T&A_6 +{\Delta t}^{-1}A_1+	A_{10}^{n,(m)}&-A_7\\
A_4^T& A_5^T&A_7^T&-A_8
\end{bmatrix}
\left[ {\begin{array}{*{20}{c}}
	\bm\alpha^{n,(m)}\\
	\bm\beta^{n,(m)}\\
	\bm\gamma^{n,(m)}\\
	\bm\zeta^{n,(m)}
	\end{array}} \right]
=\bm {\widetilde  b},
\end{align}
where 
\begin{align}
\bm {\widetilde  b} =  G'(\bm x_n^{(m-1)}) \bm x_n^{(m-1)} - G(\bm x_n^{(m-1)}).
\end{align}

\begin{remark}
	The global matrix does need to be updated at each time step and each iteration; however, the Jacobian matrix can be obtained simply by multiplying $A_1$ by $\mathrm{diag}(\mathcal F'(\bm \gamma_n^{(m-1)})$. Therefore, the computation is reduced greatly compared to standard HDG.
\end{remark}

\subsection{Local Solver}
\label{sec:local_solver}

One of the main advantages of the HDG methods compared to other DG methods is that we can
{\em locally eliminate} the unknowns $\bm \alpha_n^{(m)}$, $\bm \beta_n^{(m)}$ and $\bm \gamma_n^{(m)}$ from the large system of equations \eqref{system_equation}. 
Let us show how to carry out the local elimination for equations of the Interpolatory HDG method.

The system  \eqref{system_equation_group1} can be rewritten as
\begin{align}\label{system_equation2}
\begin{bmatrix}
B_1 & -B_2&B_3\\
B_2^T & B_4&-B_5\\
B_3^T&B_5^T&B_6\\
\end{bmatrix}
\left[ {\begin{array}{*{20}{c}}
	\bm{x}\\
	\bm{y}\\
	\bm{z}
	\end{array}} \right]
=\left[ {\begin{array}{*{20}{c}}
	b_1\\
	b_2\\
	b_3
	\end{array}} \right],
\end{align}
where $\bm{x}=[\bm \alpha^{n,(m)};\bm \beta^{n,(m)}]$, $\bm{y}=\bm{\gamma}^{n,(m)}$, $\bm{z}=\bm{\zeta}^{n,(m)}$, $ \bm {\widetilde  b} = [ b_1;b_2;b_3] $, and $\{B_i\}_{i=1}^6$ are the corresponding blocks of the coefficient matrix in \eqref{system_equation_group1}.  The system \eqref{system_equation2} is equivalent to following equations:
\begin{subequations}\label{algebra_full}
	\begin{align}
	B_1 \bm x -B_2\bm y +B_3\bm z&= b_1,\label{algebre_full_a}\\
	B_2^T \bm x +B_4\bm y -B_5\bm z &= b_2,\label{algebre_full_b}\\
	B_3^T\bm x+ B_5^T\bm y + B_6 \bm z&=b_3.\label{algebre_full_c}
	\end{align}
\end{subequations}
We efficiently solve \eqref{algebre_full_a} and \eqref{algebre_full_b} to express $\bm x$ and $\bm y$ in terms of $\bm z$.

To do this, note that since $\bm V_h$ and $W_h$ are discontinuous finite element spaces the matrices $B_1$ and $B_4$ are both block diagonal with small blocks. Therefore, these matrices can be easily inverted, and the inverses are also block diagonal with small blocks.  Furthermore, $ B_1 $ and $ B_1^{-1} $ are both positive definite.  Also introduce
\begin{align*}
Q = B_2^TB_1^{-1}B_2+B_4.
\end{align*}
As mentioned above, $B_1$ is block diagonal with small blocks and therefore it is easy to invert.  The matrix $B_2$ not block diagonal; however, $B_2 = [ A_2, A_3 ]^T $ and $ A_2 $ and $ A_3 $ are both block diagonal with small blocks.  Therefore, $Q$ is block diagonal with small blocks and is easily inverted.  Also, since $ B_4 = A_6 +{\Delta t}^{-1}A_1+	A_{10}^{n,(m)} $ and both $ A_1 $ and $ A_6 $ are positive definite, $Q$ is guaranteed to be invertible if $ \Delta t $ is small enough or if $ F'(\gamma) \geq 0 $ for any $ \gamma $.

Now solve \eqref{algebre_full_a} and \eqref{algebre_full_b} to obtain
\begin{align}
\bm x&= B_1^{-1}B_2Q^{-1}\left((B_5+B_2^TB_1^{-1}B_3)\bm z+ b_2-B_2^TB_1^{-1}b_1\right) -B_1^{-1}B_3\bm z + B_1^{-1}b_1\nonumber\\
&=:\tilde B_1 \bm z +\tilde b_1,\label{local_eli_a}\\
\bm y &= Q^{-1} \left((B_5+B_2^TB_1^{-1}B_3)\bm z + b_2-B_2^TB_1^{-1}b_1\right)\nonumber\\
&=:\tilde B_2 \bm{\gamma}^n +\tilde b_2.\label{local_eli_b}
\end{align}
Then insert $\bm x$ and $\bm y$ into \eqref{algebre_full_c} to obtain the final system only involving $\bm z$
\begin{align}\label{local_eli_c}
(B_3^T \tilde B_1 + B_5^T \tilde B_2 +  B_6) \bm z = b_3 -B_3^T\tilde b_1 -B_5^T \tilde b_2
\end{align}

\begin{remark}
	For HDG methods, the standard approach is to first compute the local solver independently on each element and then assemble the global system.  The process we follow here is to first assemble the global system and then reduce its dimension by simple block-diagonal algebraic operations.  The two approaches are equivalent.
\end{remark}

Equations \eqref{local_eli_a}-\eqref{local_eli_b} say we can express the approximate scalar state variable and flux in terms of the approximate traces on the element boundaries.  The global equation \eqref{local_eli_c} only involves the approximate traces. Therefore, the high number of globally coupled degrees of freedom in the HDG method is significantly reduced. This is an excellent feature of HDG methods.

\section{Error Analysis}
\label{sec:analysis}

Next, we carry out an error analysis of the Interpolatory HDG method. In this first work on Interpolatory HDG, we assume the nonlinearity is globally Lipschitz, i.e., there is a constant $L>0$ such that 
\begin{align}\label{lip}
|F(\bm q, u) - F(\bm r, v)|_{\mathbb R}\le L(|\bm q-\bm r|_{\mathbb R^d} + |u-v|_{\mathbb R})
\end{align}
for all $ \bm q, \bm r \in \mathbb{R}^d $ and $ u, v \in \mathbb{R}$.  It would be interesting to 
investigate less restrictive assumptions on the nonlinearity; we leave this to be considered elsewhere.

We assume the solution of the PDE \eqref{semilinear_pde} exists and is unique for $ t \in [0,T] $.  We also assume the standard and the Interpolatory HDG equations have unique solutions on $ [0,T] $.  We assume the mesh is uniformly shape regular.  
Furthermore, for both methods we assume the projection $ P $ used for the initial condition is given by $ P = \Pi_W $, where $ \Pi_W $ is introduced below.

We adopt the standard notation $W^{m,p}(\Omega)$ for Sobolev spaces on $\Omega$ with norm $\|\cdot\|_{m,p,\Omega}$ and seminorm $|\cdot|_{m,p,\Omega}$ . When $p=2$, instead of $W^{m,2}(\Omega)$, we write $H^{m}(\Omega)$, and omit the index $p$ in the corresponding norm and seminorms. Also, we set $H_0^1(\Omega):=\{v\in H^1(\Omega):v=0 \;\mbox{on}\; \partial \Omega\}$.  Finally, we set $	H(\text{div},\Omega) := \{\bm{v}\in [L^2(\Omega)]^d, \nabla\cdot \bm{v}\in L^2(\Omega)\}.$
\subsection{Auxiliary projections}

We describe a couple of projections which will be very useful in our analysis.

We begin by introducing the projection operator $\Pi_h(\bm{q},u) := (\bm{\Pi}_{V} \bm{q},\Pi_{W}u)$ defined in \cite{MR2629996}, where 
$\bm{\Pi}_{V} \bm{q}$ and $\Pi_{W}u$ denote components of the projection of $\bm{q}$ and $u$ into $\bm{V}_h$ and $W_h$, respectively. The value of the projection on each element $K\in \mathcal{T}_h$ is determined by requiring that the components satisfy the equations
\begin{subequations}\label{HDG_projection_operator}
	\begin{align}
	(\bm\Pi_V\bm q,\bm r)_K &= (\bm q,\bm r)_K,\qquad\qquad \forall \bm r\in[\mathcal P_{k-1}(K)]^d,\label{projection_operator_1}\\
	(\Pi_Wu, w)_K &= (u, w)_K,\qquad\qquad \forall  w\in \mathcal P_{k-1}(K	),\label{projection_operator_2}\\
	\langle\bm\Pi_V\bm q\cdot\bm n+\tau\Pi_Wu,\mu\rangle_{e} &= \langle\bm q\cdot\bm n+\tau u,\mu\rangle_{e},~\;\forall \mu\in \mathcal P_{k}(e),\label{projection_operator_3}
	\end{align}
\end{subequations}
for all faces $e$ of the simplex $K$. 
We also need $ P_M $, the standard $L^2$-orthogonal projection into $M_h$, which satisfies
\begin{align}\label{standpm}
\left\langle P_M y-y, \mu\right\rangle_e = 0, \quad \forall \mu\in M_h.
\end{align}

The following lemma was established in \cite{MR2629996} and provides the approximation properties of the projection operator \eqref{HDG_projection_operator}.
\begin{lemma}\label{pro_error}
	Suppose $k\ge 0$, $\tau|_{\partial K}$ is nonnegative and $\tau_K^{\max}:=\max\tau|_{\partial K}>0$. Then the system \eqref{HDG_projection_operator} is uniquely solvable for $\bm{\Pi}_V\bm{q}$ and $\Pi_W u$. Furthermore, there is a constant $C$ independent of $K$ and $\tau$ such that 
	\begin{align}
	\|{\bm{\Pi}_V}\bm{q}-q\|_K \leq Ch_{K}^{\ell_{\bm{q}}+1}|\bm{q}|_{\bm{H}^{\ell_{\bm{q}}+1}(K)}+Ch_{K}^{\ell_{{u}}+1}\tau_{K}^{*}{|u|}_{{H}^{\ell_{{u}}+1}(K)}\label{Proerr_q}\\
	\|{{\Pi}_W}{u}-u\|_K \leq Ch_{K}^{\ell_{{u}}+1}|{u}|_{{H}^{\ell_{{u}}+1}(K)}+C\frac{h_{K}^{\ell_{{\bm{q}}}+1}}{\tau_K^{\max}}{|\nabla\cdot \bm{q}|}_{{H}^{\ell_{\bm{q}}}(K)}
	\end{align}
	for $\ell_{\bm{q}},\ell_{u}$ in $[0,k]$. Here $\tau_K^{*}:=\max\tau|_{{\partial K}\backslash F^{*}}$, where $F^{*}$ is a face of $K$ at which $\tau|_{\partial K}$ is maximum.
\end{lemma}

The second auxiliary operator comes from standard FE theory. Indeed, we have the following interpolation error estimates \cite{MR2373954}.
\begin{lemma}\label{lemmainter}
	Suppose $k\ge 0$. If $w\in C(\bar K)\cap H^{k+1}(K)$ and $\bm r\in [C(\bar K)]^d \cap[ H^{k+1}(K)]^d$, then there exists a constant $C$ independent of $K\in\mathcal T_h$ such that 
	\begin{align}\label{inter}
	\|w - \mathcal  I_h w\|_K \le Ch^{k+1} \|w\|_{k+1,K},\\
	\|\bm r - {\mathcal I}_h \bm r\|_K \le Ch^{k+1} \|\bm r\|_{k+1,K}.
	\end{align}
\end{lemma}

\subsection{Main Results}

We can now state our main result for the Interpolatory HDG method.
\begin{theorem}\label{main_err_qu}
	For all $0<t\le T$, the solution $ (\bm q_h, u_h) $ of the Interpolatory HDG equations satisfy
	\begin{align*}
	\hspace{2em}&\hspace{-2em}\|\bm q(t) - \bm q_h(t)\|_{\mathcal T_h}^2\\
	&\le C  \|(\bm\Pi_V {\bm{q}} -\bm q)(0)\|_{\mathcal{T}_h}^2 + C \int_0^t \bigg(  \|F(-\bm q, u)- \mathcal I_h F(-\bm q, u) \|_{\mathcal T_h}^2 \\
	& \quad  + \|\Pi_Wu_t-u_t\|_{\mathcal T_h}^2  +\|\bm\Pi_V {\bm{q}}_t -\bm q_t\|_{\mathcal{T}_h}^2 + \|\bm{\Pi}_V \bm q  - \bm q\|_{\mathcal T_h}^2\\
	&
	\quad+ \|{\Pi}_W u  - u\|_{\mathcal T_h}^2 + \|\mathcal I_h \bm q  - \bm q\|_{\mathcal T_h}^2 + \|\mathcal I_hu  - u\|_{\mathcal T_h}^2 \bigg),\\
	\hspace{2em}&\hspace{-2em}\|u(t) - u_h(t)\|_{\mathcal T_h}^2\\
	&\le   C \int_0^t \bigg(  \|F(-\bm q, u)- \mathcal I_h F(-\bm q, u) \|_{\mathcal T_h}^2  + \|\Pi_Wu_t-u_t\|_{\mathcal T_h}^2+ \|\bm{\Pi}_V \bm q  - \bm q\|_{\mathcal T_h}^2\\
	& \quad + \|{\Pi}_W u  - u\|_{\mathcal T_h}^2 + \|\mathcal I_h \bm q  - \bm q\|_{\mathcal T_h}^2 + \|\mathcal I_hu  - u\|_{\mathcal T_h}^2 \bigg).
	\end{align*}
\end{theorem}
The error bounds for the standard HDG method are obtained when we replace the interpolation operator $ \mathcal{I}_h $ by the identity.
By Lemma \ref{pro_error},  Lemma \ref{lemmainter}, and Theorem \ref{main_err_qu}, we can easily obtain convergence rates for smooth solutions.

\begin{corollary}
	If $ u $, $ \bm q $, and $ F(-\bm q,u) $ are sufficiently smooth for $ t \in [0,T] $, then for all $0<t\le T$ the solution $ (\bm q_h, u_h) $ of the Interpolatory HDG equations satisfy
	\begin{align*}
	\|\bm q(t) - \bm q_h(t)\|_{\mathcal T_h}&\le C h^{k+1},\\
	\|u(t) - u_h(t)\|_{\mathcal T_h}&\le C h^{k+1}.
	\end{align*}
\end{corollary}


%
%

\section{Proof of the error estimates}
Here, we prove the main result for the Interpolatory HDG method, Theorem 4.1.  Our proof relies on techniques used in \cite{ChabaudCockburn12,MR3403707,MR1068202}.  We proceed in several steps.

\subsection{Step 1: Equations for the Projection of the Errors}
\begin{lemma}\label{error_u}
	For $\varepsilon_h^{\bm q} = \bm\Pi_V\bm q-\bm q_h $, $ \varepsilon_h^u=\Pi_Wu-u_h $, and $ \varepsilon_h^{\widehat u}=P_Mu-\widehat u_h$, we have
	\begin{subequations}\label{err_eq}
		\begin{align}
		(\varepsilon_h^{\bm q},\bm r)_{\mathcal T_h}-(\varepsilon_h^u,\nabla\cdot\bm r)_{\mathcal T_h}+\langle\varepsilon_h^{\widehat u},\bm r\cdot\bm n\rangle_{\partial\mathcal T_h} &= (\bm\Pi_V\bm q-\bm q,\bm r)_{\mathcal T_h},\label{err_eq_a}\\
		(\partial_t\varepsilon_h^u,w)_{\mathcal T_h}-(\varepsilon_h^{\bm q} ,\nabla w)_{\mathcal{T}_h}+\langle{\varepsilon}_h^{\widehat {\bm q}}\cdot\bm n,w\rangle_{\partial{\mathcal{T}_h}} \quad &\nonumber\\
		+ (F(-\bm q,u)-\mathcal I_hF(-\bm q_h, u_h),w)_{\mathcal T_h} &=(\Pi_Wu_t-u_t,w)_{\mathcal T_h},\label{err_eq_b}\\
		\langle{\varepsilon}_h^{\widehat {\bm q}}\cdot \bm{n} ,\mu\rangle_{\partial{\mathcal{T}_h}\backslash \varepsilon_h^{\partial}} &= 0,\label{err_eq_c}\\
		\varepsilon_h^u|_{t=0}&=0,\label{err_eq_d}
		\end{align}
	\end{subequations}
	for all $(\bm r,w,\mu)\in \bm V_h\times W_h\times M_h$, where
	\begin{align}\label{numerical_flux}
	{\varepsilon}_h^{\widehat {\bm q}} \cdot\bm n =\varepsilon_h^{\bm q}\cdot \bm n +\tau(\varepsilon_h^u-\varepsilon_h^{\widehat u}) \quad \text{on}\;\partial \mathcal T_h.
	\end{align}
\end{lemma}
\begin{proof}
	Let us begin by noting that the exact solution $(\bm q,u)$ satisfies
	\begin{align*}
	(\bm{q},\bm{r})_{\mathcal{T}_h}-(u,\nabla\cdot \bm{r})_{\mathcal{T}_h}+\langle u,\bm{r}\cdot\bm n \rangle_{\partial{\mathcal{T}_h}} &= 0, \\
	(u_t,w)_{\mathcal T_h}-(\bm{q},\nabla w)_{\mathcal{T}_h}+\langle{\bm{q}}\cdot \bm{n},w\rangle_{\partial{\mathcal{T}_h}} + (F(-\bm q, u),w)_{\mathcal T_h} &= (f,w)_{\mathcal{T}_h},
	\end{align*}
	for all $\bm r\in\bm V_h$ and $w\in W_h$.  Since $P_M$ is the $L^2$-projection into $M_h$, it satisfies the orthogonality property
	\begin{align}\label{PM_projection}
	\langle \tau(P_Mu - u),\mu\rangle_{\partial \mathcal T_h} = 0 \quad \text{for all} \;\mu\in M_h
	\end{align}
	because $\tau$ is piecewise constant on $\partial \mathcal T_h$.  By this orthogonality property and the orthogonality properties \eqref{projection_operator_1} and \eqref{projection_operator_2} of $\bm\Pi_V$ and $\Pi_W$, respectively, we have
	\begin{align*}
	(\bm\Pi_V\bm{q},\bm{r})_{\mathcal{T}_h}-(\Pi_W u,\nabla\cdot \bm{r})_{\mathcal{T}_h}+\langle P_M u,\bm r\cdot \bm n\rangle_{\partial{\mathcal{T}_h}} &= (\bm{\Pi}_V\bm q-\bm q,\bm r)_{\mathcal T_h}, \\
	(\Pi_Wu_t,w)_{\mathcal T_h}-(\bm\Pi_V\bm{q},\nabla w)_{\mathcal{T}_h}+(F(-\bm q, u),w)_{\mathcal T_h} \quad &\\
	+\langle \bm\Pi_V {\bm{q}}\cdot \bm{n}      
	+\tau (\Pi_W u-P_Mu),w\rangle_{\partial{\mathcal{T}_h}}& =(\Pi_Wu_t-u_t + f,w)_{\mathcal T_h},
	\end{align*}
	for all $\bm r\in\bm V_h$ and $w\in W_h$. Subtracting the first two equations defining the HDG method, \eqref{semi_discretion_standard_a}  and \eqref{semi_discretion_standard_b}, from the above two equations, respectively, we readily obtain \eqref{err_eq_a} and \eqref{err_eq_b}.
	
	To prove  \eqref{err_eq_c} we proceed as follows. By the definition of $\varepsilon_h^{\widehat {\bm q}}$ in \eqref{numerical_flux},
	\begin{align*}
	\langle \varepsilon_h^{\widehat {\bm q}} \cdot \bm n,\mu \rangle_{\partial \mathcal T_h\backslash\varepsilon_h^{\partial}} = \langle (\bm \Pi_V \bm q - \bm q_h)\cdot\bm n + \tau (\Pi_W u - u_h - P_M u + \widehat u_h),\mu  \rangle_{\partial\mathcal T_h\backslash\varepsilon_h^{\partial}}.
	\end{align*}
	Hence, by the orthogonality property \eqref{projection_operator_3} of the projection $\bm\Pi_V$, $\Pi_W$ and property \eqref{PM_projection} of the projection $P_M$, we obtain
	\begin{align*}
	\langle \varepsilon_h^{\widehat {\bm q}} \cdot \bm n,\mu \rangle_{\partial \mathcal T_h\backslash\varepsilon_h^{\partial}} &= \langle ( \bm q - \bm q_h)\cdot\bm n + \tau ( u - u_h -  u + \widehat u_h),\mu  \rangle_{\partial\mathcal T_h\backslash\varepsilon_h^{\partial}}\\
	&= \langle  \bm q \cdot\bm n,\mu  \rangle_{\partial\mathcal T_h\backslash\varepsilon_h^{\partial}} -  \langle  \widehat {\bm q}_h \cdot\bm n,\mu  \rangle_{\partial\mathcal T_h\backslash\varepsilon_h^{\partial}},
	\end{align*}
	and equation \eqref{err_eq_c} follows since both of the above terms are zero. Indeed, the first is equal to zero because $\bm q$ is in $H(\text{div},\Omega)$ and the second because the normal component of $\bm q_h$ is single-valued by equation \eqref{semi_discretion_standard_c} defining the HDG method.
	
	It remains to prove equation \eqref{err_eq_d}. By equation \eqref{semi_discretion_standard_d} defining the HDG method, $u_h|_{t=0} = P u_0 = \Pi_W u_0$, and so
	\begin{align*}
	\varepsilon_h^u|_{t=0} = \Pi_W u_0 - u_h|_{t=0} = \Pi_W u_0 -\Pi_W u_0=0.
	\end{align*}
	This completes the proof.
\end{proof}

\subsection{Step 2: Estimate of $\varepsilon_h^u$ in $L^{\infty}(L^2)$ by an Energy Argument}
\label{sec:energy_argument_u}

\begin{lemma}\label{energy_norm}
	For any $t>0$, we have 
	\begin{align*}
	\hspace{1em}&\hspace{-1em}\frac 1 2 \|\varepsilon_h^u(t)\|_{\mathcal T_h}^2 +\int_0^t (\|\varepsilon_h^{\bm q}\|_{\mathcal T_h}^2 +\|\sqrt{\tau}(\varepsilon_h^u-\varepsilon_h^{\widehat u})\|_{\partial{\mathcal{T}_h}}^2) \\
	& =\int_0^t  (\bm\Pi_V {\bm{q}} -\bm q, \varepsilon_h^{\bm q})_{\mathcal{T}_h} 
	+(\Pi_Wu_t-u_t,\varepsilon_h^u)_{\mathcal T_h} - (F(-\bm q, u) - \mathcal I_h F(-\bm q_h, u_h),\varepsilon_h^u)_{\mathcal T_h}.
	\end{align*}
\end{lemma}
\begin{proof}
	Taking $\bm{r} = \varepsilon_h^{\bm q}$ in \eqref{err_eq_a}, $w = \varepsilon_h^u$ in \eqref{err_eq_b}, $\mu= -\varepsilon_h^{\widehat u}$ in \eqref{err_eq_c}, adding the resulting three equations, and noting that $\mu = 0$ on $\varepsilon_h^{\partial}$, we obtain
	\begin{align*}
	\frac 1 2 \frac {d}{dt} \|\varepsilon_h^u\|_{\mathcal T_h} + \|\varepsilon_h^{\bm q}\|_{\mathcal T_h}^2 +\Theta  &= (\bm\Pi_V {\bm{q}} -\bm q, \varepsilon_h^{\bm q})_{\mathcal{T}_h} 
	+(\Pi_Wu_t-u_t,\varepsilon_h^u)_{\mathcal T_h}\\
	& \quad - (F(-\bm q, u) - \mathcal I_h F(-\bm q_h, u_h),\varepsilon_h^u)_{\mathcal T_h},
	\end{align*}
	where 
	\begin{align*}
	\Theta &= - (\varepsilon_h^u, \nabla\cdot \varepsilon_h^{\bm q})_{\mathcal T_h} +\langle\varepsilon_h^{\widehat u},\varepsilon_h^{\bm q}\cdot \bm n\rangle_{\partial\mathcal T_h}- (\varepsilon_h^{\bm q} ,\nabla \varepsilon_h^u)_{\mathcal{T}_h} \\
	& \quad + \langle\varepsilon_h^{\bm q}\cdot \bm{n} +\tau (\varepsilon_h^u-\varepsilon_h^{\widehat u}),\varepsilon_h^u \rangle_{\partial {\mathcal{T}_h}}-\langle\varepsilon_h^{\bm q}\cdot \bm{n} +\tau (\varepsilon_h^u-\varepsilon_h^{\widehat u}),\varepsilon_h^{\widehat u} \rangle_{\partial{\mathcal{T}_h}}\\
	&= \langle \tau (\varepsilon_h^u-\varepsilon_h^{\widehat u}),\varepsilon_h^u-\varepsilon_h^{\widehat u}\rangle_{\partial{\mathcal{T}_h}}.
	\end{align*}
	Here, we used the definition of $\varepsilon_h^{\widehat{\bm q}}$ in \eqref{numerical_flux} and integrated by parts. The desired identity follows after integrating in time over the interval $(0,t)$ and using the fact that $\varepsilon_h^u(0) = 0$ by \eqref{err_eq_d}.
\end{proof}

\subsection{Step 3: Norms associated to the interpolation operator $\mathcal{I}_h$}

To estimate the error in the nonlinear term in the Interpolatory HDG method, we utilize the following auxiliary norms on $ W_h $ and $ \bm V_h $:
\begin{align}\label{def_e_norm}
\|w\|_h = \left[\sum_{K\in\mathcal T_h}\sum_{i=1}^{\ell_K} |w(\xi^K_i)|^2 h^d_K\right]^{1/2},  ~  \|\bm r\|_h = \left[\sum_{K\in\mathcal T_h}\sum_{i=1}^{\ell_K} \sum_{s=1}^d |r_s(\xi^K_i)|^2 h^d_K\right]^{1/2},
\end{align}
for any $w\in W_h$ and $\bm r\in \bm V_h$, where $r_s$ is the $s$-th component of $\bm r$ and $ \{ \xi_i^K \} $ are the FE nodes as in Section \ref{sec:GHDG_semidiscrete}.  Here, $h_K$ denotes the diameter of the element $K$. These norms are very similar to the auxiliary norms on continuous FE spaces introduced in \cite{MR3403707,MR1068202}.  
The following lemma is fundamental for our analysis.
\begin{lemma}\label{lemma_eq_norm}
	There exist two positive constants $c_1$ and $c_2$ independent of $h$ such that 
	\begin{align}\label{eq_norm}
	c_1 \|w\|_h \le \|w\|_{\mathcal T_h} \le c_2 \|v\|_h,\\
	c_1 \|\bm r\|_h \le \|\bm r\|_{\mathcal T_h} \le c_2 \|\bm r\|_h,
	\end{align}
	for all $w\in W_h$ and $\bm r\in \bm V_h$.
\end{lemma}

The proof of this lemma is essentially given in \cite{MR1068202}; we provide the details for the sake of completeness.
\begin{proof}
	We only prove the first inequality; the second is similar. Let $\widehat K$ be a reference element and let $\mathcal P^k(\widehat K)$ be the space of polynomials of degree up to $k$ defined on $\widehat K$. Since  $\mathcal P^k(\widehat K)$ is finite dimensional, there exist positive constants $\hat{c}_1$ and $\hat{c}_2$ depending only on $k$ such that
	\begin{align}\label{in_ref}
	\hat{c}_1\sum_{i=1}^n |p(\widehat  \xi_i)|^2 \le \int_{\widehat K} |p|^2 \le \hat{c}_2\sum_{i=1}^n |p(\widehat \xi_i)|^2,
	\end{align}
	for all $p\in \mathcal P^k(\widehat K)$, where $\{\widehat \xi_i\}_{i=1}^n$ are the nodal points on the reference element.
	
	Now for $ K \in \mathcal T_h $, let $w\in W_h(K)$ and set $ x = \mathbb F \widehat x =  \mathbb B \widehat x + b$, where $\mathbb F$ is the affine mapping from the reference element to $ K $. Since  $w|_K \circ \mathbb F \in\mathcal P^k(\widehat K) $, we obtain 
	that
	\begin{align*}
	\hat{c}_1\sum_{i=1}^{\ell_K} |w(\xi_i^K)|^2 \le  |\mathrm{det} \mathbb{B}|^{-1} \int_{ K} |w|^2 \le \hat{c}_2\sum_{i=1}^{\ell_K} |w(\xi_i^K)|^2.
	\end{align*}
	Since the mesh is uniformly shape regular, there exist positive constants $ \hat{c}_3 $ and $ \hat{c}_4 $ depending only on the regularity constant of the mesh such that $\hat{c}_3 h_K^d \leq |\mathrm{det} \mathbb B| \leq \hat{c}_4 h_K^d$. This implies that
	\begin{align*}
	\hat{c}_1\hat{c}_3 \sum_{i=1}^{\ell_K} |w(\xi_i^K)|^2\,h^d_K \le  \int_{ K} |w|^2 \le \hat{c}_2\hat{c}_4\sum_{i=1}^{\ell_K} |w(\xi_i^K)|^2\,h^d_K,
	\end{align*}
	and the result follows with $c_1^2:=\hat{c}_1\hat{c}_3$ and $c_2^2:=\hat{c}_2\hat{c}_4$. This completes the proof.
\end{proof}

\subsection{Step 4: Estimate of the nonlinear term}
The crucial component in the analysis is estimating the error in the nonlinear term.  We decompose $F(-\bm q, u)-\mathcal I_hF(-\bm q_h, u_h)$ as 
\begin{align*}
F(-\bm q, u)&-\mathcal I_hF(-\bm q_h, u_h) \\
&= F(-\bm q, u)- \mathcal I_h F(-\bm q, u)    + \mathcal I_h F(-\bm q, u)  -  \mathcal I_h F(-\bm \Pi_V\bm q, \Pi_W u)  \\
& \quad + \mathcal I_h F(-\bm \Pi_V\bm q, \Pi_W u)  -\mathcal I_hF(-\bm q_h, u_h)\\
&=: R_1 + R_2 + R_3.
\end{align*}
The first term $R_1$ can be bounded by the standard FE interpolation error \eqref{inter} in Lemma \ref{lemmainter} due to the smoothness assumption for $F(-\bm q, u)$. For the  terms $R_2$ and $R_3$, we have the following estimates.
\begin{lemma}
	We have 
	\begin{align*}
	\|\mathcal I_h F(-\bm q, u)  -  \mathcal I_h F(-\bm \Pi_V\bm q, \Pi_W u) \|_{\mathcal T_h}&\le \frac {Lc_2}{c_1}(\|\bm \Pi_V \bm q - \bm q\|_{\mathcal T_h} + \| \bm q - \mathcal I_h\bm q\|_{\mathcal T_h}) \\
	& \quad +\frac {Lc_2}{c_1}( \| \Pi_W u- u\|_{\mathcal T_h} + \|  u- \mathcal I_h u\|_{\mathcal T_h}),\\
	\| \mathcal I_h F(-\bm \Pi_V\bm q, \Pi_W u)  -\mathcal I_hF(-\bm q_h, u_h)\|_{\mathcal T_h}  &\le \frac{Lc_2}{c_1}(\|\bm \Pi_V \bm q - \bm q_h\|_{\mathcal T_h} + \| \Pi_W u- u_h\|_{\mathcal T_h}).
	\end{align*}
\end{lemma}
\begin{proof}
	We prove the first inequality; the second is similar.  From inequality  \eqref{eq_norm} in Lemma \ref{lemma_eq_norm}, we have
	\begin{align*}
	\hspace{2em}&\hspace{-2em}\|\mathcal I_h F(-\bm q, u)  -  \mathcal I_h F(-\bm \Pi_V\bm q, \Pi_W u) \|_{\mathcal T_h}\\
	&\le  {c_2}\|\mathcal I_hF(-\bm q, u)-\mathcal I_hF(-\bm \Pi_V \bm q, \Pi_W u)\|_h\\
	&=  {c_2}\|F(-\bm q, u)-F(-\bm \Pi_V \bm q, \Pi_W u)\|_h\\
	&\le  Lc_2(\|\bm \Pi_V \bm q - \bm q\|_h + \| \Pi_W u- u\|_h)\\
	&= Lc_2 (\|\bm \Pi_V \bm q - \mathcal I_h\bm q\|_h + \| \Pi_W u- \mathcal I_h u\|_h)\\
	&\le \frac {Lc_2}{c_1}(\|\bm \Pi_V \bm q - \mathcal I_h\bm q\|_{\mathcal T_h} + \| \Pi_W u- \mathcal I_h u\|_{\mathcal T_h})\\
	&\le \frac {Lc_2}{c_1}(\|\bm \Pi_V \bm q - \bm q\|_{\mathcal T_h} + \| \bm q - \mathcal I_h\bm q\|_{\mathcal T_h} + \| \Pi_W u- u\|_{\mathcal T_h} + \|  u- \mathcal I_h u\|_{\mathcal T_h}).
	\end{align*}
\end{proof}


To prove the main result for the Interpolatory HDG method, we use the following integral Gronwall inequality, which can be found in \cite{MR2431403}.
\begin{lemma}\label{con_gr_ineq}
	Let $f, g, h$ be piecewise continuous nonnegative functions defined on $(a, b)$.  If $g$ is nondecreasing and there is a positive constant $C$ independent of $t$ such that
	\begin{align*}
	\forall t\in (a,b),\quad f(t)+h(t)\le g(t) + C\int_a^t f(s) ds,
	\end{align*}
	then
	\begin{align*}
	\forall t\in (a,b),\quad f(t)+h(t)\le e^{C(t-a)}g(t).
	\end{align*}
\end{lemma}

\subsection{Step 5: Estimate of $\varepsilon^u_h$}

\begin{theorem}\label{theorem_err_u}
	We have 
	\begin{align*}
	\|\varepsilon_h^u(t)\|_{\mathcal T_h}^2 +\int_0^t (\|\varepsilon_h^{\bm q}\|_{\mathcal T_h}^2 +2\|\sqrt{\tau}(\varepsilon_h^u-\varepsilon_h^{\widehat u})\|_{\partial{\mathcal{T}_h}}^2)\le  e^{\mathcal Lt}  \int_0^t \mathcal M,
	\end{align*}
	where 
	\begin{align*}
	\mathcal M &= \|F(-\bm q, u)- \mathcal I_h F(-\bm q, u) \|_{\mathcal T_h}^2  + \|\Pi_Wu_t-u_t\|_{\mathcal T_h}^2 + 2\|\bm{\Pi}_V \bm q  - \bm q\|_{\mathcal T_h}^2\\
	& \quad + \frac{4L^2c_2^2}{c_1^2}(\|\bm{\Pi}_V \bm q  - \bm q\|_{\mathcal T_h}^2 + \|{\Pi}_W u  - u\|_{\mathcal T_h}^2 + \|\mathcal I_h \bm q  - \bm q\|_{\mathcal T_h}^2 + \|\mathcal I_hu  - u\|_{\mathcal T_h}^2),\\
	\mathcal L  &= \frac{2Lc_2}{c_1}+\frac{2L^2c_2^2}{c_1^2} + 3.
	\end{align*}
\end{theorem}

\begin{proof} 
	Apply the Cauchy-Schwarz inequality to each term of the right-hand side of the identity in Lemma \ref{energy_norm} to get
	\begin{align*}
	(\bm\Pi_V {\bm{q}} -\bm q, \varepsilon_h^{\bm q})_{\mathcal{T}_h}  
	&\le   \|\bm\Pi_V {\bm{q}} -\bm q\|_{\mathcal{T}_h}^2 + \frac 1 4 \|\varepsilon_h^{\bm q}\|_{\mathcal{T}_h}^2,\\
	(\Pi_Wu_t-u_t,\varepsilon_h^u)_{\mathcal T_h} 
	&\le  \frac 1 2 \|\Pi_Wu_t-u_t\|_{\mathcal T_h}^2 + \frac 1 2  \|\varepsilon_h^u\|_{\mathcal T_h}^2,\\
	%
	(F(-\bm q, u) - \mathcal I_h F(-\bm q_h, u_h),\varepsilon_h^u)_{\mathcal T_h}
	&\le  (1+\frac{Lc_2}{c_1}+\frac{L^2c_2^2}{c_1^2}) \|\varepsilon_h^u\|_{\mathcal T_h}^2  + \frac{1}{4} \|\varepsilon_h^{\bm q}\|_{\mathcal T_h}^2+\mathcal N,
	\end{align*}
	where 
	\begin{align*}
	\mathcal N &= \frac1 2 \|F(-\bm q, u)- \mathcal I_h F(-\bm q, u) \|_{\mathcal T_h}^2 \\
	& \quad + \frac{2L^2c_2^2}{c_1^2}(\|\bm{\Pi}_V \bm q  - \bm q\|_{\mathcal T_h}^2 + \|{\Pi}_W u  - u\|_{\mathcal T_h}^2 + \|\mathcal I_h \bm q  - \bm q\|_{\mathcal T_h}^2 + \|\mathcal I_hu  - u\|_{\mathcal T_h}^2).
	\end{align*}
	Lemma \ref{energy_norm} implies
	\begin{align*}
	&\|\varepsilon_h^u(t)\|_{\mathcal T_h}^2 +\int_0^t (\|\varepsilon_h^{\bm q}\|_{\mathcal T_h}^2 +2\|\sqrt{\tau}(\varepsilon_h^u-\varepsilon_h^{\widehat u})\|_{\partial{\mathcal{T}_h}}^2)\le   \int_0^t \mathcal M + \mathcal L \int_0^t \|\varepsilon_h^u(t)\|_{\mathcal T_h}^2,
	\end{align*}
	where 
	\begin{align*}
	\mathcal M &= 2\mathcal N + \|\Pi_Wu_t-u_t\|_{\mathcal T_h}^2 +2 \|\bm{\Pi}_V \bm q  - \bm q\|_{\mathcal T_h}^2\\
	&= \|F(-\bm q, u)- \mathcal I_h F(-\bm q, u) \|_{\mathcal T_h}^2  + \|\Pi_Wu_t-u_t\|_{\mathcal T_h}^2 + 2\|\bm{\Pi}_V \bm q  - \bm q\|_{\mathcal T_h}^2\\
	& \quad + \frac{4L^2c_2^2}{c_1^2}(\|\bm{\Pi}_V \bm q  - \bm q\|_{\mathcal T_h}^2 + \|{\Pi}_W u  - u\|_{\mathcal T_h}^2 + \|\mathcal I_h \bm q  - \bm q\|_{\mathcal T_h}^2 + \|\mathcal I_hu  - u\|_{\mathcal T_h}^2),\\
	\mathcal L  &= \frac{2Lc_2}{c_1}+\frac{2L^2c_2^2}{c_1^2} + 3.
	\end{align*}
	The integral Gronwall inequality in Lemma \ref{con_gr_ineq} gives the result.
\end{proof}

\subsection{Step 6: Estimate of $\varepsilon_h^q$ in $L^{\infty}(L^2)$ by an energy argument}

\begin{theorem}\label{error_ana_q}
	We have
	\begin{align*}
	\|\varepsilon_h^{\bm q}(t)\|_{\mathcal T_h}^2 +\|\sqrt{\tau}(\varepsilon_h^u(t)-\varepsilon_h^{\widehat u}(t))\|_{\partial{\mathcal{T}_h}}^2\le  C \bigg(\|\bm\Pi_V {\bm{q}} -\bm q)(0)\|_{\mathcal{T}_h}^2+ \int_0^t \mathcal G\bigg).
	\end{align*}
	where 
	\begin{align*}
	\mathcal G &=\|F(-\bm q, u)- \mathcal I_h F(-\bm q, u) \|_{\mathcal T_h}^2+ \|\bm\Pi_V {\bm{q}}_t -\bm q_t\|_{\mathcal{T}_h}^2 +\|\Pi_Wu_t-u_t\|_{\mathcal T_h}^2\\
	& \quad + \|\bm{\Pi}_V \bm q  - \bm q\|_{\mathcal T_h}^2 + \|{\Pi}_W u  - u\|_{\mathcal T_h}^2 + \|\mathcal I_h \bm q  - \bm q\|_{\mathcal T_h}^2 + \|\mathcal I_hu  - u\|_{\mathcal T_h}^2.
	\end{align*}
\end{theorem}
\begin{proof}
	To prove this result, we use a slightly different set of equations for the projection of the errors than the equations in Lemma \ref{error_u}. We keep all of the error equations except \eqref{err_eq_a}, which we replace by the equation obtained by differentiating \eqref{err_eq_a} with respect to time:
	\begin{subequations}\label{err_eq_2}
		\begin{align}
		(\partial_t\varepsilon_h^{\bm q},\bm r)_{\mathcal T_h}-(\partial_t\varepsilon_h^u,\nabla\cdot\bm r)_{\mathcal T_h}+\langle \partial_t\varepsilon_h^{\widehat u},\bm r\cdot\bm n\rangle_{\partial\mathcal T_h} &= (\bm\Pi_V\bm q_t-\bm q_t,\bm r)_{\mathcal T_h},\label{err_eq_a2}\\
		(\partial_t\varepsilon_h^u,w)_{\mathcal T_h}-(\varepsilon_h^{\bm q} ,\nabla w)_{\mathcal{T}_h}+\langle{\varepsilon}_h^{\widehat {\bm q}}\cdot\bm n,w\rangle_{\partial{\mathcal{T}_h}} \quad &\nonumber\\
		+ (F(-\bm q,u)-\mathcal I_hF(-\bm q_h, u_h),w)_{\mathcal T_h} &=(\Pi_Wu_t-u_t,w)_{\mathcal T_h},\label{err_eq_b2}\\
		\langle{\varepsilon}_h^{\widehat {\bm q}}\cdot \bm{n} ,\mu\rangle_{\partial{\mathcal{T}_h}\backslash \varepsilon_h^{\partial}} &= 0,\label{err_eq_c2}\\
		\varepsilon_h^u|_{t=0}&=0,\label{err_eq_d2}
		\end{align}
	\end{subequations}
	for all $(\bm r,w,\mu)\in \bm V_h\times W_h\times M_h$, where ${\varepsilon}_h^{\widehat {\bm q}} \cdot\bm n =\varepsilon_h^{\bm q}\cdot \bm n +\tau(\varepsilon_h^u-\varepsilon_h^{\widehat u}) $ on $\partial \mathcal T_h$. 
	
	Next,  take  $\bm r=\varepsilon_h^{\bm q} $ in \eqref{err_eq_a2}, $w=\partial_t\varepsilon_h^u$  in \eqref{err_eq_b2}, and $\mu=-\partial_t\varepsilon_h^{\widehat u}$  in \eqref{err_eq_c2} to obtain
	\begin{align*}
	\|\partial_t\varepsilon_h^u\|^2_{\mathcal T_h}+(\partial_t\varepsilon_h^{\bm q},\varepsilon_h^{\bm q})_{\mathcal T_h}+\Theta &=(\bm\Pi_V\bm q_t-\bm q_t,\varepsilon_h^{\bm q})_{\mathcal T_h}+(\Pi_Wu_t-u_t,\partial_t\varepsilon_h^u)_{\mathcal T_h}\\
	& \quad -  (F(-\bm q, u)-\mathcal I_hF(-\bm q_h, u_h),\partial_t\varepsilon_h^u)_{\mathcal T_h},
	\end{align*}
	where
	\begin{align*}
	\Theta &=-(\partial_t\varepsilon_h^u,\nabla\cdot\varepsilon_h^{\bm q})_{\mathcal T_h}+\langle\partial_t\varepsilon_h^{\widehat u},\varepsilon_h^{\bm q}\cdot\bm n\rangle_{\partial\mathcal T_h}-(\varepsilon^n_{\bm q},\nabla(\partial_t\varepsilon_h^u))_{\mathcal T_h}\\
	& \quad +\langle \varepsilon_h^{\widehat {\bm q}}\cdot\bm n,\partial_t\varepsilon_h^u\rangle_{\partial\mathcal T_h} -\langle \varepsilon_h^{\widehat {\bm q}}\cdot\bm n,\partial_t\varepsilon_h^{\widehat u}\rangle_{\partial\mathcal T_h}\\
	&=\tau\langle\varepsilon_h^u-\varepsilon_h^{\widehat u},\partial_t\varepsilon_h^u-\partial_t\varepsilon_h^{\widehat u}\rangle_{\partial\mathcal T_h}.
	\end{align*}
	Here, we used the definition of $\varepsilon_h^{\widehat{\bm q}}$ in \eqref{numerical_flux} and integrated by parts. Integrating in time over the interval $(0,t)$ gives the following identity:
	\begin{align*}
	\hspace{2em}&\hspace{-2em}\frac 1 2 [\|\varepsilon_h^{\bm q}(t)\|_{\mathcal T_h}^2 +\|\sqrt{\tau}(\varepsilon_h^u(t)-\varepsilon_h^{\widehat u}(t))\|_{\partial{\mathcal{T}_h}}^2]+\int_0^t \|\partial_t\varepsilon_h^{u}\|_{\mathcal T_h}^2 \\
	&= \frac 12 [\|\varepsilon_h^{\bm q}(0)\|_{\mathcal T_h}^2 +\|\sqrt{\tau}(\varepsilon_h^u(0)-\varepsilon_h^{\widehat u}(0))\|_{\partial{\mathcal{T}_h}}^2 ]\\
	& \quad +\int_0^t  (\bm\Pi_V {\bm{q}_t} -\bm q_t, \varepsilon_h^{\bm q})_{\mathcal{T}_h} 
	+(\Pi_Wu_t-u_t,\partial_t\varepsilon_h^u)_{\mathcal T_h} \\
	& \quad - \int_0^t  (F(-\bm q, u) - \mathcal I_h F(-\bm q_h, u_h),\partial_t\varepsilon_h^u)_{\mathcal T_h}.
	\end{align*}
	Applying the Cauchy-Schwarz inequality to each term of the right-hand side of the above identity gives
	\begin{align*}
	(\bm\Pi_V {\bm{q}}_t -\bm q_t, \varepsilon_h^{\bm q})_{\mathcal{T}_h}  
	&\le  \frac 1 2 \|\bm\Pi_V {\bm{q}}_t -\bm q_t\|_{\mathcal{T}_h}^2 + \frac 1 2 \|\varepsilon_h^{\bm q}\|_{\mathcal{T}_h}^2,\\
	(\Pi_Wu_t-u_t,\partial_t\varepsilon_h^u)_{\mathcal T_h} 
	&\le  \|\Pi_Wu_t-u_t\|_{\mathcal T_h}^2 + \frac 1 4  \|\partial_t\varepsilon_h^u\|_{\mathcal T_h}^2,\\
	(F(-\bm q, u) - \mathcal I_h F(-\bm q_h, u_h),\partial_t\varepsilon_h^u)_{\mathcal T_h} &=\frac{3}{4} \|\partial_t\varepsilon_h^u\|_{\mathcal T_h}^2+ \frac{2L^2c_2^2}{c_1^2} (\|\varepsilon_h^u\|_{\mathcal T_h}^2 +  \|\varepsilon_h^{\bm q}\|_{\mathcal T_h}^2)+\mathcal K,
	\end{align*}
	where
	\begin{align*}
	\mathcal K &=\|F(-\bm q, u)- \mathcal I_h F(-\bm q, u) \|_{\mathcal T_h}^2\\  & \quad + \frac{4L^2c_2^2}{c_1^2}(\|\bm{\Pi}_V \bm q  - \bm q\|_{\mathcal T_h}^2 + \|{\Pi}_W u  - u\|_{\mathcal T_h}^2 + \|\mathcal I_h \bm q  - \bm q\|_{\mathcal T_h}^2 + \|\mathcal I_hu  - u\|_{\mathcal T_h}^2).
	\end{align*}
	This implies
	\begin{align*}
	\hspace{1em}&\hspace{-1em} \|\varepsilon_h^{\bm q}(t)\|_{\mathcal T_h}^2 +\|\sqrt{\tau}(\varepsilon_h^u(t)-\varepsilon_h^{\widehat u}(t))\|_{\partial{\mathcal{T}_h}}^2 \\
	&\le   [\|\varepsilon_h^{\bm q}(0)\|_{\mathcal T_h}^2 +\|\sqrt{\tau}(\varepsilon_h^u(0)-\varepsilon_h^{\widehat u}(0))\|_{\partial{\mathcal{T}_h}}^2 ] + \int_0^t  \mathcal G \\
	& \quad + \frac{4L^2c_2^2}{c_1^2} \int_0^t \|\varepsilon_h^u\|_{\mathcal T_h}^2
	+  \mathcal H  \int_0^t \|\varepsilon_h^{\bm q}(t)\|_{\mathcal T_h}^2,
	\end{align*}
	where 
	\begin{align*}
	\mathcal G &= 2\mathcal K + \|\bm\Pi_V {\bm{q}}_t -\bm q_t\|_{\mathcal{T}_h}^2 +2\|\Pi_Wu_t-u_t\|_{\mathcal T_h}^2\\
	&=2\|F(-\bm q, u)- \mathcal I_h F(-\bm q, u) \|_{\mathcal T_h}^2+ \|\bm\Pi_V {\bm{q}}_t -\bm q_t\|_{\mathcal{T}_h}^2 +2\|\Pi_Wu_t-u_t\|_{\mathcal T_h}^2\\
	& \quad + \frac{8L^2c_2^2}{c_1^2}(\|\bm{\Pi}_V \bm q  - \bm q\|_{\mathcal T_h}^2 + \|{\Pi}_W u  - u\|_{\mathcal T_h}^2 + \|\mathcal I_h \bm q  - \bm q\|_{\mathcal T_h}^2 + \|\mathcal I_hu  - u\|_{\mathcal T_h}^2),\\
	\mathcal H &= \frac{4L^2c_2^2}{c_1^2} + 1.
	\end{align*}
	Apply the integral Gronwall inequality in Lemma \ref{con_gr_ineq} to obtain
	\begin{align*}
	\hspace{2em}&\hspace{-2em}\|\varepsilon_h^{\bm q}(t)\|_{\mathcal T_h}^2 +\|\sqrt{\tau}(\varepsilon_h^u(t)-\varepsilon_h^{\widehat u}(t))\|_{\partial{\mathcal{T}_h}}^2\\
	&\le  e^{\mathcal Ht} \left( [\|\varepsilon_h^{\bm q}(0)\|_{\mathcal T_h}^2 +\|\sqrt{\tau}(\varepsilon_h^u(0)-\varepsilon_h^{\widehat u}(0))\|_{\partial{\mathcal{T}_h}}^2 ] + \int_0^t  \mathcal G  + \frac{4L^2c_2^2}{c_1^2} \int_0^t \|\varepsilon_h^u\|_{\mathcal T_h}^2\right).
	\end{align*}

	Next, differentiate the equation in Lemma \ref{energy_norm} and evaluate the result at $t=0$ to obtain
	\begin{align*}
	\|\varepsilon_h^{\bm q}(0)\|_{\mathcal T_h}^2 +\|\sqrt{\tau}(\varepsilon_h^u-\varepsilon_h^{\widehat u})(0)\|_{\partial{\mathcal{T}_h}}^2 &= ((\bm\Pi_V {\bm{q}} -\bm q)(0), \varepsilon_h^{\bm q}(0))_{\mathcal{T}_h},
	\end{align*}
	since $\varepsilon_h^{u}(0) = 0$. This implies that 
	\begin{align*}
	\|\varepsilon_h^{\bm q}(0)\|_{\mathcal T_h}^2 +\|\sqrt{\tau}(\varepsilon_h^u-\varepsilon_h^{\widehat u})(0)\|_{\partial{\mathcal{T}_h}}^2 \le   \|\bm\Pi_V {\bm{q}} -\bm q)(0)\|_{\mathcal{T}_h}^2.
	\end{align*}
	Since $\|\varepsilon_h^u\|$ has been estimated in Theorem \ref{theorem_err_u}, we have
	\begin{align*}
	\|\varepsilon_h^{\bm q}(t)\|_{\mathcal T_h}^2 +\|\sqrt{\tau}(\varepsilon_h^u(t)-\varepsilon_h^{\widehat u}(t))\|_{\partial{\mathcal{T}_h}}^2\le  C \bigg(\|\bm\Pi_V {\bm{q}} -\bm q)(0)\|_{\mathcal{T}_h}^2+ \int_0^t \mathcal G\bigg).
	\end{align*}
\end{proof}

This completes the proof of our main result, Theorem \ref{main_err_qu}.

\section{Numerical Results}
\label{sec:numerics}

In this section, we consider three examples chosen to demonstrate the performance of the {Interpolatory HDG} method.  The domain is the unit square $\Omega = [0,1]\times [0,1]\subset \mathbb R^2$ in 2D and the unit cube $\Omega = [0,1]\times [0,1]\times [0,1]\subset \mathbb R^3$ in 3D. Backward Euler is applied for the time discretization and the time step is chosen as $\Delta t = h^{k+1}$, where $k$ is the degree of polynomial.  The $ L^2 $ projection is used for the initial data.  We report the errors at the final time $ T = 1 $ for polynomial degrees $ k = 0 $ and $ k = 1 $.

We consider the following examples:
\begin{description}
	\item[Example 1] A reaction diffusion equation (the Allen-Cahn or Chafee-Infante equation): The nonlinear term is $F(\nabla u, u) = u^3-u$ and the source term $f$ is chosen so that the exact solution is $u = \sin(t)\sin(\pi x)\sin(\pi y)$ in 2D and $u = \sin(t)\sin(\pi x)\sin(\pi y)\sin(\pi z)$ in 3D.
	\item[Example 2] A PDE from stochastic optimal control \cite{MR2597943}: The nonlinear term is $F(\nabla u, u) = |\nabla u|^2$ and the source term $f$ is chosen so that the exact solution is $u = e^{-t}\sin(\pi x)\sin(\pi y)$ in 2D.
	\item[Example 3] A scalar Burger's equation: The nonlinear term is $F(\nabla u, u) =[u,u]^T\cdot\nabla u$. The source term $f$ is chosen so that $u = e^{-t}\sin(\pi x)\sin(\pi y)$ is the exact solution in 2D.
	%
\end{description}
We present 2D Interpolatory HDG results for all examples with $ k = 0 $ and $ k = 1 $; we also display the corresponding results for the standard HDG results for comparison.  Finally, we give 3D Interpolatory HDG numerical results for the reaction diffusion equation when $k=1$.  The results are shown in Table \ref{table:GHDG_2D_CI_eqn}--Table \ref{table:GHDG_2D_Burgers}.  For all examples, the Interpolatory HDG method converges at the optimal rate.  Furthermore, for the 2D reaction diffusion equation, the errors for the Interpolatory HDG method are similar to the standard HDG method when $ k = 1 $.  As indicated previously, the standard HDG is equivalent to the Interpolatory HDG when $ k = 0 $ as the numerical results indicate.

Note that the nonlinearities of these examples are not globally Lipschitz, as assumed in our theoretical results.  Even for this more difficult case, we observe the same optimal orders of convergence as the ones predicted by the theory for the globally Lipschitz case.


\begin{table}
	\begin{center}
		\begin{tabular}{|c|c|c|c|c|c|}
			\hline
			$k$ & Mesh &$\norm{\bm {q}-\bm{q}_{h}}_{\mathcal T_h}$& Order& $\norm{{u}-{u}_h}_{\mathcal T_h}$ &Order  \\
			\hline
			\multirow{5}{*}{0} &256 & 3.78e-1 & -  & 1.57e-1   &  -  \\
			\cline{2-6}
			& 1024 & 1.93e-1& 0.97 & 8.43e-2& 0.89\\
			\cline{2-6}
			& 4096 &  9.72e-2& 0.99 & 4.32e-02& 0.96\\
			\cline{2-6}
			& 16384 & 4.88e-2&  0.99 & 2.19e-02& 0.98 \\
			\cline{2-6}
			& 65536 & 2.44e-2& 1.00 & 1.10e-02&  0.99\\
			\hline
			\multirow{5}{*}{1} &256 & 3.21e-2& -  & 1.94e-2 &  -  \\
			\cline{2-6}
			& 1024 & 7.91e-3& 2.02 & 4.96e-3& 1.97\\
			\cline{2-6}
			& 4096 &  1.97e-3& 2.00 & 1.24e-3& 2.00\\
			\cline{2-6}
			& 16384 & 4.92e-4&  2.00 & 3.13e-4& 2.00 \\
			\cline{2-6}
			& 65536 & 1.23e-4& 2.00 & 7.82e-5&  2.00\\
			\hline
		\end{tabular}
	\end{center}
	\caption{\label{table:GHDG_2D_CI_eqn}Interpolatory HDG Method for the 2D reaction diffusion equation}
\end{table}

\begin{table}
	\begin{center}
		\begin{tabular}{|c|c|c|c|c|c|}
			\hline
			$k$ & Mesh &$\norm{\bm {q}-\bm {q}_{h}}_{\mathcal T_h}$& Order& $\norm{{u}-{u}_h}_{\mathcal T_h}$ &Order  \\
			\hline
			\multirow{5}{*}{0} &256 & 3.78e-1 & -  & 1.57e-1   &  -  \\
			\cline{2-6}
			& 1024 & 1.93e-1& 0.97 & 8.43e-2& 0.89\\
			\cline{2-6}
			& 4096 &  9.72e-2& 0.99 & 4.32e-02& 0.96\\
			\cline{2-6}
			& 16384 & 4.88e-2&  0.99 & 2.19e-02& 0.98 \\
			\cline{2-6}
			& 65536 & 2.44e-2& 1.00 & 1.10e-02&  0.99\\
			\hline
			\multirow{5}{*}{1} &256 & 2.98e-2& -  & 1.96e-2 &  -  \\
			\cline{2-6}
			& 1024 & 7.57e-3& 2.02 & 4.97e-3& 1.97\\
			\cline{2-6}
			& 4096 &  1.91e-3& 2.00 & 1.25e-3& 2.00\\
			\cline{2-6}
			& 16384 & 4.78e-4&  2.00 & 3.12e-4& 2.00 \\
			\cline{2-6}
			& 65536 & 1.23e-4& 2.00 & 7.82e-5&  2.00\\
			\hline
		\end{tabular}
	\end{center}
	\caption{Standard HDG Method for the 2D reaction diffusion equation}
\end{table}

\begin{table}
	\begin{center}
		\begin{tabular}{|c|c|c|c|c|c|}
			\hline
			$k$ & Mesh &$\norm{\bm {q}-\bm{q}_{h}}_{\mathcal T_h}$& Order& $\norm{{u}-{u}_h}_{\mathcal T_h}$ &Order  \\
			\hline
			\multirow{5}{*}{1} &48 & 1.56e-1  & -  & 7.89e-2    &  -  \\
			\cline{2-6}
			& 384 & 4.60e-2& 1.77& 2.35e-2& 1.75\\
			\cline{2-6}
			& 3072 & 1.31e-2& 1.82 & 6.20e-3& 1.93\\
			\cline{2-6}
			&  24576 & 3.40e-3&  1.94 & 1.58e-4& 1.98 \\
			\cline{2-6}
			& 196608 & 8.24e-4& 2.05 & 3.90e-5&  1.99\\
			\hline
		\end{tabular}
	\end{center}
	\caption{\label{table:GHDG_3D_CI_eqn}Interpolatory HDG Method for the 3D reaction diffusion equation}
\end{table}


%
\begin{table}
	\begin{center}
		\begin{tabular}{|c|c|c|c|c|c|}
			\hline
			$k$ & Mesh &$\norm{\bm{q}-\bm{q}_{h}}_{\mathcal T_h}$& Order& $\norm{{u}-{u}_h}_{\mathcal T_h}$ &Order  \\
			\hline
			\multirow{5}{*}{0} &256 & 1.11e-2 & -  & 7.58e-3   &  -  \\
			\cline{2-6}
			& 1024 & 5.31e-3& 1.06 &  3.32e-3& 1.20\\
			\cline{2-6}
			& 4096 &  2.67e-3& 1.00 & 1.59e-3& 1.06\\
			\cline{2-6}
			& 16384 & 1.32e-3&  1.01 & 7.73e-04& 1.04 \\
			\cline{2-6}
			& 65536 &  6.60e-4& 1.00 & 3.83e-04&  1.02\\
			\hline
			\multirow{5}{*}{1} &256 & 1.64e-3& -  & 3.97e-4&  -  \\
			\cline{2-6}
			& 1024 & 3.42e-4& 2.26 & 8.47e-5& 2.22\\
			\cline{2-6}
			& 4096 &  8.57e-5& 2.00 & 2.11e-5& 2.00\\
			\cline{2-6}
			& 16384 & 2.14e-5&  2.00 &5.28e-6& 2.00 \\
			\cline{2-6}
			& 65536 & 5.36e-6& 2.00 & 1.32e-6&  2.00\\
			\hline
		\end{tabular}
	\end{center}
	\caption{Interpolatory HDG Method for a 2D PDE from stochastic optimal control}
\end{table}


\begin{table}
	\begin{center}
		\begin{tabular}{|c|c|c|c|c|c|}
			\hline
			Degree & Mesh &$\norm{\bm{q}-\bm{q}_{h}}_{\mathcal T_h}$& Order& $\norm{{u}-{u}_h}_{\mathcal T_h}$ &Order  \\
			\hline
			\multirow{5}{*}{0} &256 & 1.57e-1 & -  & 1.10e-1   &  -  \\
			\cline{2-6}
			& 1024 & 7.75e-2& 1.01 &  5.15e-2& 1.10\\
			\cline{2-6}
			& 4096 &  3.88e-2& 1.00 & 2.50e-02& 1.04\\
			\cline{2-6}
			& 16384 & 1.94e-2&  1.00 & 1.23e-02& 1.02 \\
			\cline{2-6}
			& 65536 & 9.69e-3& 1.00 & 6.11e-03&  1.01\\
			\hline
			\multirow{5}{*}{1} &256 & 3.21e-2& -  & 1.94e-2 &  -  \\
			\cline{2-6}
			& 1024 & 7.91e-3& 2.02 & 4.96e-3& 1.97\\
			\cline{2-6}
			& 4096 &  1.97e-3& 2.00 & 1.24e-3& 2.00\\
			\cline{2-6}
			& 16384 & 4.92e-4&  2.00 & 3.13e-4& 2.00 \\
			\cline{2-6}
			& 65536 & 1.23e-4& 2.00 & 7.81e-4&  2.00\\
			\hline
		\end{tabular}
	\end{center}
	\caption{\label{table:GHDG_2D_Burgers}Interpolatory HDG Method for 2D Burger's equation}
\end{table}



\section{Conclusion}

We proposed the Interpolatory HDG method for approximating the solution of scalar parabolic semilinear PDEs.  The Interpolatory HDG method replaces the nonlinear term with an elementwise interpolation, which leads to a simple and efficient implementation.  Specifically, unlike the standard HDG method, the Interpolatory HDG method does not require numerical quadrature to form the global matrix at each time step and at each step in a Newton iteration.  We also proved optimal convergence rates for the flux $\bm q$ and the primary unknown $u$ assuming the nonlinearity is globally Lipschitz.

Numerical experiments in 2D and 3D demonstrated that the Interpolatory HDG method converged at the optimal rates, and gave similar errors to the standard HDG method.  However, for Interpolatory HDG we did not numerically observe superconvergence by post-processing.  This is one disadvantage of Interpolatory HDG compared to standard HDG.  However, due to the computational efficiency of the interpolatory approach, Interpolatory HDG using a higher order polynomial degree may be a competitive alternative to standard HDG.  Furthermore, it may be possible to obtain superconvergence for Interpolatory HDG using an alternative post-processing approach.  We leave these issues to be thoroughly explored elsewhere.

Although we have only used simplicial elements and the spaces given by \eqref{spaces}, our analysis extends in a straightforward manner to the HDG and mixed methods (new and old) obtained in the theory of M-decompositions,
see \cite{CockburnFuSayasM17}. Thus, in 2D, polygonal elements of any shape can be used, see \cite{CockburnFuM2D17}, and in 3D, tetrahedral, prismatic, pyramidal or hexagonal elements, see \cite{CockburnFuM3D17}. Indeed, for these methods, an auxiliary projection $\Pi_h(\bm{q},u)$, see 
its general definition in \cite[Definition 3.1]{CockburnFuSayasM17} and its approximation properties in \cite[Proposition 3.4]{CockburnFuSayasM17}, with which the error analysis becomes identical to the one we have presented.

The implementation of Interpolatory HDG in Section \ref{sec:GHDG_implementation} easily extends to these other HDG and mixed methods only in certain situations.  We plan to further explore implementation and superconvergence issues for other Interpolatory HDG and mixed methods in the future.

The idea leading to the Interpolatory HDG method can be applied to many other types of nonlinear PDEs.  We plan to investigate the Interpolatory HDG method for complex nonlinear PDE systems in the future.

\section*{Acknowledgements}
J.\ Singler and Y.\ Zhang were supported in part by National Science Foundation grant DMS-1217122. J.\ Singler and Y.\ Zhang thank the IMA for funding research visits, during which
some of this work was completed. Y.\ Zhang thanks Zhu Wang for many valuable conversations.

\appendix

\section{Implementation details for General Nonlinearities}

\label{sec:GHDG_general_implementation}

\subsection{The Interpolatory HDG formulation}

The full Interpolatory HDG discretization is to find $(\bm q^n_h,u^n_h,\widehat u^n_h)\in \bm V_h\times W_h\times M_h$ such that
\begin{equation}\label{full_discretion_Group_im2}
\begin{split}
(\bm{q}^n_h,\bm{r})_{\mathcal{T}_h}-(u^n_h,\nabla\cdot \bm{r})_{\mathcal{T}_h}+\left\langle\widehat{u}^n_h,\bm r \cdot \bm n \right\rangle_{\partial{\mathcal{T}_h}} &= 0, \\
(\partial^+_tu^n_h,w)_{\mathcal T_h}+(\nabla\cdot\bm{q}^n_h, w)_{\mathcal{T}_h}+\langle\tau(u_h^n - \widehat u_h^n),w\rangle_{\partial{\mathcal{T}_h}} + ( \mathcal I_h F(-\bm q_h^n, u_h^n),w)_{\mathcal{T}_h}&= (f^n,w)_{\mathcal{T}_h},\\
\langle {\bm{q}}^n_h\cdot \bm{n} + \tau(u_h^n - \widehat u_h^n), \mu\rangle_{\partial{\mathcal{T}_h}\backslash\varepsilon^{\partial}_h} &=0,\\
u^0_h &=\Pi_W u_0,
\end{split}
\end{equation}
for all $(\bm r,w,\mu)\in \bm V_h\times W_h\times M_h$ and $n=1,2,\ldots,N$. Similar to Section \ref{sec:GHDG_implementation}, we have
\begin{align}
( \mathcal I_h F(-\bm q_h^n, u_h^n),w)_{\mathcal{T}_h} = A_1 \mathcal F(\bm \alpha^n, \bm\beta^{n},\bm\gamma^{n}),
\end{align}
where 
\begin{align}
\mathcal F(\bm\alpha^{n},\bm\beta^{n},\bm\gamma^{n}) = [F(\alpha_1^{n},\beta_1^{n}, \gamma_1^{n}),\ldots,F(\alpha_{N_1}^{n},\beta_{N_1}^{n}, \gamma_{N_1}^{n})]^T.
\end{align}
Then the system \eqref{full_discretion_Group_im2} can be rewritten as
\begin{align}\label{system_equation_group32}
\underbrace{\begin{bmatrix}
	A_1 & 0 &-A_2  & A_4 \\
	0 & A_1 &-A_3  & A_5 \\
	A_2^T& A_3^T&A_6 +{\Delta t}^{-1}A_1&-A_7\\
	A_4^T& A_5^T&A_7^T&-A_8
	\end{bmatrix}}_{M}
\underbrace{\left[ {\begin{array}{*{20}{c}}
		\bm\alpha^{n}\\
		\bm\beta^{n}\\
		\bm\gamma^{n}\\
		\bm\zeta^{n}
		\end{array}} \right]}_{\bm x_{n}}+
\underbrace{\left[ {\begin{array}{*{20}{c}}
		0\\
		0\\
		A_1 \mathcal F(\bm\alpha^{n},\bm\beta^{n},\bm\gamma^{n})\\
		0
		\end{array}} \right]}_{\mathscr F(\bm x_{n})}
=\underbrace{\left[ {\begin{array}{*{20}{c}}
		0\\
		0\\
		b_1^n+{\Delta t}^{-1}A_1\bm\gamma^{n-1} \\
		0
		\end{array}} \right]}_{\bm b_n},
\end{align}
i.e., $ M\bm x_n +  \mathscr F(\bm x_n) = \bm b_n $.

Newton's method proceeds as in Section \ref{sec:GHDG_implementation}, but the Jacobian matrix $G'(\bm x_n^{(m-1)})$ is now given by%
\begin{align*}
G'(\bm x_n^{(m-1)}) = M+\mathscr F'(\bm x_n^{(m-1)}),  \quad  \mathscr F'(\bm x_n^{(m-1)}) = \begin{bmatrix}
0 & 0 &0  & 0 \\
0 & 0 &0  & 0 \\
A_{11}^{n,(m)}& A_{12}^{n,(m)}&A_{13}^{n,(m)}&0\\
0 & 0 &0  & 0 
\end{bmatrix},
\end{align*}
where for $ k = 1, 2, 3, $ we define
\begin{align*}
A_{1k}^{n,(m)} &= A_1\text{diag}(\mathcal F_k'(\bm{\alpha}^{n,(m-1)}),\bm{\beta}^{n,(m-1)}),\bm{\gamma}^{n,(m-1)}),\\
\mathcal F_k'(\bm\alpha^{n},\bm\beta^{n},\bm\gamma^{n}) &= [F_k'(\alpha_1^{n,(m-1)},\beta_1^{n,(m-1)}, \gamma_1^{n,(m-1)}),\cdots,F_k'(\alpha_{N_1}^{n,(m-1)},\beta_{N_1}^{n,(m-1)}, \gamma_{N_1}^{n,(m-1)})]^T.
\end{align*}
Therefore, the linear system that must be solved is now given by
\begin{align}\label{system_equation_group12}
\begin{bmatrix}
A_1 & 0 &-A_2  & A_4 \\
0 & A_1 &-A_3  & A_5 \\
A_2^T+	A_{11}^{n,(m)}& A_3^T+	A_{12}^{n,(m)}&A_6 +{\Delta t}^{-1}A_1+	A_{13}^{n,(m)}&-A_7\\
A_4^T& A_5^T&A_7^T&-A_8
\end{bmatrix}
\left[ {\begin{array}{*{20}{c}}
	\bm\alpha^{n,(m)}\\
	\bm\beta^{n,(m)}\\
	\bm\gamma^{n,(m)}\\
	\bm\zeta^{n,(m)}
	\end{array}} \right]
=\bm {\widetilde  b},
\end{align}
where 
\begin{align}
\bm {\widetilde  b} =  G'(\bm x_n^{(m-1)}) \bm x_n^{(m-1)} - G(\bm x_n^{(m-1)}).
\end{align}
\subsection{Local Solver}

The system  \eqref{system_equation_group12}  can be rewritten as
\begin{align}\label{system_equation22}
\begin{bmatrix}
B_1 & B_2&B_3\\
B_4 & B_5&-B_6\\
B_3^T&B_6^T&B_7\\
\end{bmatrix}
\left[ {\begin{array}{*{20}{c}}
	\bm{x}\\
	\bm{y}\\
	\bm{z}
	\end{array}} \right]
=\left[ {\begin{array}{*{20}{c}}
	b_1\\
	b_2\\
	b_3
	\end{array}} \right],
\end{align}
where $\bm{x}=[\bm{\alpha^{n,(m)}};\bm{\beta^{n,(m)}}]$, $\bm{y}=\bm{\gamma}^{n,(m)}$, $\bm{z}=\bm{\zeta}^{n,(m)}$, $ \bm {\widetilde  b} = [ b_1;b_2;b_3] $, and $\{B_i\}_{i=1}^7$ are the corresponding blocks of the coefficient matrix in \eqref{system_equation_group12}.  The system \eqref{system_equation22} is equivalent with following equations:
\begin{subequations}\label{algebra_full2}
	\begin{align}
	B_1 \bm x + B_2\bm y +B_3\bm z&= b_1,\label{algebre_full_a2}\\
	B_4 \bm x +B_5\bm y -B_6\bm z &= b_2,\label{algebre_full_b2}\\
	B_3^T\bm x+ B_6^T\bm y + B_7 \bm z&=b_3.\label{algebre_full_c2}
	\end{align}
\end{subequations}

Similar to before, the matrices $B_1$ and $B_5$ are block diagonal with small blocks and they can be easily inverted.  Use \eqref{algebre_full_a2} and \eqref{algebre_full_b2} to express $\bm x$ and $\bm y$ in terms of $\bm z$ as follows:
\begin{align}
\bm x &= B_1^{-1}B_2\left(B_4B_1^{-1}B_2+B_5\right)^{-1}\left((B_6+B_4B_1^{-1}B_3)\bm z+ b_2-B_4B_1^{-1}b_1\right) -B_1^{-1}B_3\bm z + B_1^{-1}b_1\nonumber\\
&=:\tilde B_1 \bm z +\tilde b_1,\label{local_eli_a2}\\
\bm y &=\left(B_4B_1^{-1}B_2+B_5\right)^{-1} \left((B_6+B_4B_1^{-1}B_3)\bm z + b_2-B_4B_1^{-1}b_1\right)\nonumber\\
&=:\tilde B_2 \bm{\gamma}^n +\tilde b_2\label{local_eli_b2},
\end{align}
where 
\begin{align*}
Q = B_4B_1^{-1}B_2+B_5 = B_4B_1^{-1}B_2 +A_6 +{\Delta t}^{-1}A_1+	A_{13}^{n,(m)}.
\end{align*}
As in Section \ref{sec:local_solver}, the matrix $Q$ is block diagonal with small blocks.  Since $A_1$ is positive definite, if $\Delta t$ is small enough then $ Q $ is easily inverted.  Then we insert $\bm x$ and $\bm y$ into \eqref{algebre_full_c} and obtain the final system only involving $\bm z$:
\begin{align}\label{local_eli_c2}
(B_3^T \tilde B_1 + B_5^T \tilde B_2 +  B_6) \bm z = b_3 -B_3^T\tilde b_1 -B_5^T \tilde b_2
\end{align}

\bibliographystyle{plain}
\bibliography{yangwen_ref_papers,more}

\begin{thebibliography}{10}

\bibitem{MR2373954}
Susanne~C. Brenner and L.~Ridgway Scott.
\newblock {\em The mathematical theory of finite element methods}, volume~15 of
  {\em Texts in Applied Mathematics}.
\newblock Springer, New York, third edition, 2008.

\bibitem{MR3626531}
Aycil Cesmelioglu, Bernardo Cockburn, and Weifeng Qiu.
\newblock Analysis of a hybridizable discontinuous {G}alerkin method for the
  steady-state incompressible {N}avier-{S}tokes equations.
\newblock {\em Math. Comp.}, 86(306):1643--1670, 2017.

\bibitem{ChabaudCockburn12}
Brandon Chabaud and Bernardo Cockburn.
\newblock Uniform-in-time superconvergence of {HDG} methods for the heat
  equation.
\newblock {\em Math. Comp.}, 81(277):107--129, 2012.

\bibitem{MR1030644}
Chuan~Miao Chen, Stig Larsson, and Nai~Ying Zhang.
\newblock Error estimates of optimal order for finite element methods with
  interpolated coefficients for the nonlinear heat equation.
\newblock {\em IMA J. Numer. Anal.}, 9(4):507--524, 1989.

\bibitem{MR1172090}
Zhangxin Chen and Jim Douglas, Jr.
\newblock Approximation of coefficients in hybrid and mixed methods for
  nonlinear parabolic problems.
\newblock {\em Mat. Apl. Comput.}, 10(2):137--160, 1991.

\bibitem{MR641309}
I.~Christie, D.~F. Griffiths, A.~R. Mitchell, and J.~M. Sanz-Serna.
\newblock Product approximation for nonlinear problems in the finite element
  method.
\newblock {\em IMA J. Numer. Anal.}, 1(3):253--266, 1981.

\bibitem{CockburnDurham16}
Bernardo Cockburn.
\newblock Static condensation, hybridization, and the devising of the {HDG}
  methods.
\newblock In G.R. Barrenechea, F.~Brezzi, A.~Cagniani, and E.H. Georgoulis,
  editors, {\em Building Bridges: Connections and Challenges in Modern
  Approaches to Numerical Partial Differential Equations}, volume 114 of {\em
  Lect. Notes Comput. Sci. Engrg.}, pages 129--177. Springer Verlag, Berlin,
  2016.
\newblock LMS Durham Symposia funded by the London Mathematical Society.
  Durham, U.K., on July 8--16, 2014.

\bibitem{CockburnFuM2D17}
Bernardo Cockburn and Guosheng Fu.
\newblock Superconvergence by {$M$}-decompositions. {P}art {II}: {C}onstruction
  of two-dimensional finite elements.
\newblock {\em ESAIM Math. Model. Numer. Anal.}, 51(1):165--186, 2017.

\bibitem{CockburnFuM3D17}
Bernardo Cockburn and Guosheng Fu.
\newblock Superconvergence by {$M$}-decompositions. {P}art {III}:
  {C}onstruction of three-dimensional finite elements.
\newblock {\em ESAIM Math. Model. Numer. Anal.}, 51(1):365--398, 2017.

\bibitem{CockburnFuSayasM17}
Bernardo Cockburn, Guosheng Fu, and Francisco-Javier. Sayas.
\newblock Superconvergence by {$M$}-decompositions. {P}art {I}: {G}eneral
  theory for {HDG} methods for diffusion.
\newblock {\em Math. Comp.}, 86(306):1609--1641, 2017.

\bibitem{MR2485455}
Bernardo Cockburn, Jayadeep Gopalakrishnan, and Raytcho Lazarov.
\newblock Unified hybridization of discontinuous {G}alerkin, mixed, and
  continuous {G}alerkin methods for second order elliptic problems.
\newblock {\em SIAM J. Numer. Anal.}, 47(2):1319--1365, 2009.

\bibitem{MR2629996}
Bernardo Cockburn, Jayadeep Gopalakrishnan, and Francisco-Javier Sayas.
\newblock A projection-based error analysis of {HDG} methods.
\newblock {\em Math. Comp.}, 79(271):1351--1367, 2010.

\bibitem{MR3463051}
Bernardo Cockburn and Jiguang Shen.
\newblock A hybridizable discontinuous {G}alerkin method for the
  {$p$}-{L}aplacian.
\newblock {\em SIAM J. Sci. Comput.}, 38(1):A545--A566, 2016.

\bibitem{MR2587427}
Benjamin~T. Dickinson and John~R. Singler.
\newblock Nonlinear model reduction using group proper orthogonal
  decomposition.
\newblock {\em Int. J. Numer. Anal. Model.}, 7(2):356--372, 2010.

\bibitem{MR0502033}
Jim Douglas, Jr. and Todd Dupont.
\newblock The effect of interpolating the coefficients in nonlinear parabolic
  {G}alerkin procedures.
\newblock {\em Math. Comput.}, 20(130):360--389, 1975.

\bibitem{MR2597943}
Lawrence~C. Evans.
\newblock {\em Partial differential equations}, volume~19 of {\em Graduate
  Studies in Mathematics}.
\newblock American Mathematical Society, Providence, RI, second edition, 2010.

\bibitem{MR702221}
C.~A.~J. Fletcher.
\newblock The group finite element formulation.
\newblock {\em Comput. Methods Appl. Mech. Engrg.}, 37(2):225--244, 1983.

\bibitem{MR798845}
C.~A.~J. Fletcher.
\newblock Time-splitting and the group finite element formulation.
\newblock In {\em Computational techniques and applications: {CTAC}-83
  ({S}ydney, 1983)}, pages 517--532. North-Holland, Amsterdam, 1984.

\bibitem{GaticaSequeira15}
Gabriel~N. Gatica and Fil\'ander~A. Sequeira.
\newblock Analysis of an augmented {HDG} method for a class of
  quasi-{N}ewtonian {S}tokes flows.
\newblock {\em J. Sci. Comput.}, 65(3):1270--1308, 2015.

\bibitem{KabariaLewCockburn15}
Hardik Kabaria, Adrian~J. Lew, and B.~Cockburn.
\newblock A hybridizable discontinuous {G}alerkin formulation for non-linear
  elasticity.
\newblock {\em Comput. Methods Appl. Mech. Engrg.}, 283:303--329, 2015.

\bibitem{MR3178584}
Dongho Kim, Eun-Jae Park, and Boyoon Seo.
\newblock Two-scale product approximation for semilinear parabolic problems in
  mixed methods.
\newblock {\em J. Korean Math. Soc.}, 51(2):267--288, 2014.

\bibitem{MR973559}
Stig Larsson, Vidar Thom\'ee, and Nai~Ying Zhang.
\newblock Interpolation of coefficients and transformation of the dependent
  variable in finite element methods for the nonlinear heat equation.
\newblock {\em Math. Methods Appl. Sci.}, 11(1):105--124, 1989.

\bibitem{MR967844}
J.~C. L\'opez~Marcos and J.~M. Sanz-Serna.
\newblock Stability and convergence in numerical analysis. {III}. {L}inear
  investigation of nonlinear stability.
\newblock {\em IMA J. Numer. Anal.}, 8(1):71--84, 1988.

\bibitem{MoroNguyenPeraireSCL12}
D.~andN.-C.~Nguyen Moro and J.~Peraire.
\newblock A hybridized discontinuous {P}etrov-{G}alerkin scheme for scalar
  conservation laws.
\newblock {\em Internat. J. Numer. Methods Engrg.}, 91:950--970, 2012.

\bibitem{NguyenPeraireCM12}
N.~C. Nguyen and J.~Peraire.
\newblock Hybridizable discontinuous {G}alerkin methods for partial
  differential equations in continuum mechanics.
\newblock {\em J. Comput. Phys.}, 231:5955--5988, 2012.

\bibitem{MR2558780}
N.~C. Nguyen, J.~Peraire, and B.~Cockburn.
\newblock An implicit high-order hybridizable discontinuous {G}alerkin method
  for nonlinear convection-diffusion equations.
\newblock {\em J. Comput. Phys.}, 228(23):8841--8855, 2009.

\bibitem{NguyenPeraireCockburnHDGAIAAINS10}
N.~C. Nguyen, J.~Peraire, and B.~Cockburn.
\newblock A hybridizable discontinuous {G}alerkin method for the incompressible
  {N}avier-{S}tokes equations ({AIAA P}aper 2010-362).
\newblock In {\em Proceedings of the 48th AIAA Aerospace Sciences Meeting and
  Exhibit}, Orlando, Florida, January 2010.

\bibitem{NguyenPeraireCockburn11}
N.~C. Nguyen, J.~Peraire, and B.~Cockburn.
\newblock An implicit high-order hybridizable discontinuous {G}alerkin method
  for the incompressible {N}avier-{S}tokes equations.
\newblock {\em J. Comput. Phys.}, 230(4):1147--1170, 2011.

\bibitem{NguyenPeraireCockburnEDG15}
N.~C. Nguyen, J.~Peraire, and B.~Cockburn.
\newblock A class of embedded discontinuous {G}alerkin methods for
  computational fluid dynamics.
\newblock {\em J. Comput. Phys.}, 302:674--692, 2015.

\bibitem{PeraireNguyenCockburnHDGAIAACNS10}
J.~Peraire, N.~C. Nguyen, and B.~Cockburn.
\newblock A hybridizable discontinuous {G}alerkin method for the compressible
  {E}uler and {N}avier-{S}tokes equations ({AIAA P}aper 2010-363).
\newblock In {\em Proceedings of the 48th AIAA Aerospace Sciences Meeting and
  Exhibit}, Orlando, Florida, January 2010.

\bibitem{MR2431403}
B\'eatrice Rivi\`ere.
\newblock {\em Discontinuous {G}alerkin methods for solving elliptic and
  parabolic equations}, volume~35 of {\em Frontiers in Applied Mathematics}.
\newblock Society for Industrial and Applied Mathematics (SIAM), Philadelphia,
  PA, 2008.
\newblock Theory and implementation.

\bibitem{MR731213}
J.~M. Sanz-Serna and L.~Abia.
\newblock Interpolation of the coefficients in nonlinear elliptic {G}alerkin
  procedures.
\newblock {\em SIAM J. Numer. Anal.}, 21(1):77--83, 1984.

\bibitem{MR1068202}
Yves Tourigny.
\newblock Product approximation for nonlinear {K}lein-{G}ordon equations.
\newblock {\em IMA J. Numer. Anal.}, 10(3):449--462, 1990.

\bibitem{UeckermannLermusiaux16}
M.~P. Ueckermann and P.~F.~J. Lermusiaux.
\newblock Hybridizable discontinuous {G}alerkin projection methods for
  {N}avier-{S}tokes and {B}oussinesq equations.
\newblock {\em J. Comput. Phys.}, 306:390--421, 2016.

\bibitem{MR2752869}
Cheng Wang.
\newblock Convergence of the interpolated coefficient finite element method for
  the two-dimensional elliptic sine-{G}ordon equations.
\newblock {\em Numer. Methods Partial Differential Equations}, 27(2):387--398,
  2011.

\bibitem{MR3403707}
Zhu Wang.
\newblock Nonlinear model reduction based on the finite element method with
  interpolated coefficients: semilinear parabolic equations.
\newblock {\em Numer. Methods Partial Differential Equations},
  31(6):1713--1741, 2015.

\bibitem{MR2112661}
Ziqing Xie and Chuanmiao Chen.
\newblock The interpolated coefficient {FEM} and its application in computing
  the multiple solutions of semilinear elliptic problems.
\newblock {\em Int. J. Numer. Anal. Model.}, 2(1):97--106, 2005.

\bibitem{MR2273051}
Zhiguang Xiong and Chuanmiao Chen.
\newblock Superconvergence of rectangular finite element with interpolated
  coefficients for semilinear elliptic problem.
\newblock {\em Appl. Math. Comput.}, 181(2):1577--1584, 2006.

\bibitem{MR2294957}
Zhiguang Xiong and Chuanmiao Chen.
\newblock Superconvergence of triangular quadratic finite element with
  interpolated coefficients for semilinear parabolic equation.
\newblock {\em Appl. Math. Comput.}, 184(2):901--907, 2007.

\bibitem{MR2391691}
Zhiguang Xiong, Yanping Chen, and Yan Zhang.
\newblock Convergence of {FEM} with interpolated coefficients for semilinear
  hyperbolic equation.
\newblock {\em J. Comput. Appl. Math.}, 214(1):313--317, 2008.

\end{thebibliography}

\end{document}